\documentclass[12pt]{amsart}
\usepackage{latexsym,amssymb,amsmath}
\usepackage{amsmath,amsfonts,eufrak}
\usepackage{epsfig}
\usepackage{graphicx}
\usepackage{verbatim,xcolor}
\newtheorem{theorem}{Theorem}[section]
\newtheorem{lemma}[theorem]{Lemma}
\newtheorem{prop}[theorem]{Proposition}
\newtheorem{cor}[theorem]{Corollary}
\newtheorem{definition}[theorem]{Definition}
\newtheorem{rmk}[theorem]{Remark}

\allowdisplaybreaks

\newcommand{\Hmm}[1]{\leavevmode{\marginpar{\tiny%
$\hbox to 0mm{\hspace*{-0.5mm}$\leftarrow$\hss}%
\vcenter{\vrule depth 0.1mm height 0.1mm width \the\marginparwidth}%
\hbox to 0mm{\hss$\rightarrow$\hspace*{-0.5mm}}$\\\relax\raggedright #1}}}
\newcommand{\nc}{\newcommand}
\nc{\les}{\lesssim}
\nc{\ges}{\gtrsim}
\nc{\nit}{\noindent}
\nc{\nn}{\nonumber}
\nc{\D}{\partial}
\nc{\diff}[2]{\frac{d #1}{d #2}}
\nc{\diffn}[3]{\frac{d^{#3} #1}{d {#2}^{#3}}}
\nc{\pdiff}[2]{\frac{\partial #1}{\partial #2}}
\nc{\pdiffn}[3]{\frac{\partial^{#3} #1}{\partial{#2}^{#3}}}
\nc{\abs}[1] {\lvert #1 \rvert}
\nc{\cAc}{{\cal A}_c}
\nc{\cE}{{\cal E}}
\nc{\mF}{{\mathcal F}}
\nc{\cF}{{\mathcal F}} 
\nc{\cP}{{\cal P}}
\nc{\cV}{{\cal V}}
\nc{\cQ}{{\cal Q}}
\nc{\cGin}{{\cal G}_{\rm in}}
\nc{\cGout}{{\cal G}_{\rm out}}
\nc{\cO}{{\cal O}}
\nc{\Lav}{{\cal L}_{\rm av}}
\nc{\cL}{{\mathcal L}}
\nc{\cH}{{\mathcal H}}
\nc{\cB}{{\cal B}}
\nc{\cZ}{{\cal Z}}
\nc{\cR}{{\cal R}}
\nc{\cT}{{\cal T}}
\nc{\cY}{{\cal Y}}
\nc{\cX}{{\cal X}}
\nc{\cXT}{{{\cal X}(T)}}
\nc{\cBT}{{{\cal B}(T)}}
\nc{\vD}{{\vec \mathcal{D}}}
\nc{\efield}{\mathcal{E}}
\nc{\vE}{{\vec \efield}}
\nc{\vB}{{\vec \mathcal{B}}}
\nc{\vH}{{\vec \mathcal{H}}}
\nc{\ty}{{\tilde y}}
\nc{\tu}{{\tilde u}}
\nc{\tV}{{\tilde V}}
\nc{\Pc}{{\bf P_c}}
\nc{\bx}{{\bf x}}
\nc{\bX}{{\bf X}}
\nc{\bXYZ}{{\bf XYZ}}
\nc{\bY}{{\bf Y}}
\nc{\bF}{{\bf F}}
\nc{\bS}{{\bf S}}
\nc{\dV}{{\delta V}}
\nc{\dE}{{\delta E}}
\nc{\TT}{{\Theta}}
\nc{\dPsi}{{\delta\Psi}}
\nc{\order}{{\cal O}}
\nc{\Rout}{R_{\rm out}}
\nc{\eplus}{e_+}
\nc{\eminus}{e_-}
\nc{\epm}{e_\pm}
\nc{\eps}{\varepsilon}
\nc{\vnabla}{{\vec\nabla}}
\nc{\G}{\Gamma}
\nc{\w}{\omega}
\nc{\mh}{h}
\nc{\mg}{g}
\nc{\vphi}{\varphi}
\nc{\tlambda}{\tilde\lambda}
\nc{\be}{\begin{equation}}
\nc{\ee}{\end{equation}}
\nc{\ba}{\begin{eqnarray}}
\nc{\ea}{\end{eqnarray}}

\nc{\g}{\gamma}
\nc{\ol}{\overline}

  \nc{\F}{\mathcal F}
 \def\H{\mathcal H}
\def\R{\mathbb R}
\def\C{\mathbb C}
\nc{\T}{\mathbb T}
\nc{\Z}{\mathbb Z}
\nc{\N}{\mathbb N}
\nc{\pt}{\partial_t}
\nc{\la}{\langle}
\nc{\ra}{\rangle}
\nc{\infint}{\int_{-\infty}^{\infty}}
\nc{\halfwidth}{6.5cm}
\nc{\uu}{\" u}
\nc{\oo}{\" o}
\nc{\nlayers}{L} \nc{\nsectors}{M}
\nc{\indicator}{\mathbf{1}}
\nc{\Rhole}{R_{\rm hole}}
\nc{\Rring}{R_{\rm ring}}
\nc{\neff}{n_{\rm eff}}
\nc{\Frem}{F_{\rm rem}}
\nc{\DD}{\Delta}
\nc{\cD}{\mathcal D}
\nc{\lnorm}{\left\|}
\nc{\rnorm}{\right\|}
\nc{\rnormp}{\right\|_{\ell^{p,\eps}}}
\nc{\rar}{\rightarrow}
\nc{\sgn}{{\rm sign}}
\nc{\non}{\nonumber}
\nc{\wh}{\widehat}
\allowdisplaybreaks
\date{\today}

\begin{document}

\title[Zakharov System]{The Zakharov System on the Upper Half--plane}

\author[  Erdo\u{g}an and Tzirakis]{ M. B. Erdo\u{g}an   and N. Tzirakis}
 
\address{Department of Mathematics \\
University of Illinois \\
Urbana, IL 61801, U.S.A.}
\email{berdogan@illinois.edu}


\address{Department of Mathematics \\
University of Illinois \\
Urbana, IL 61801, U.S.A.}
\email{tzirakis@illinois.edu}

\begin{abstract}

In this paper we study the Zakharov system   on the upper half--plane $U=\{(x ,y)\in \R^2: y>0\}$ with non-homogenous boundary conditions. In particular we obtain low regularity local well--posedness using the restricted norm method of Bourgain and   the Fourier--Laplace method of solving initial and boundary value problems.  Moreover we prove that the nonlinear part of the solution is in a smoother space than the initial data. To our knowledge this is the first paper which establishes low regularity results for the 2d initial-boundary value Zakharov system.

\end{abstract}

\maketitle

\section{Introduction}

In this paper we study the following initial-boundary value problem for the 2d-Zakharov system
\begin{equation}\label{ZS}
\left\{
\begin{array}{l}
i u_{t}+\Delta u= nu, \quad (x,y) \in U,\, t>0,\\
n_{tt}-\Delta n=\Delta(|u|^2),\quad (x,y) \in U,\, t>0,\\
u(x,y,0)=f(x,y)\in H^{s_1}(U),\quad u(x ,0,t)=k(x ,t),  \quad x \in\R, t>0,\\
n(x,y,0)=g(x,y)\in H^{s_2}(U),\quad n_t(x,y,0)=h(x,y)\in H^{s_2-1}(U), \\
 n(x ,0,t)=\ell(x ,t),\quad x \in\R, t>0,
\end{array}
\right.
\end{equation}
where $U=\{ (x ,y)\in \R^2: y>0\}$ is the upper half--plane. For $s>-\frac12$, we define $H^s(U)$ by extension to $\R^2$.  
For $s<-\frac12$, we define it as a tempered distribution in $H^s(\R^2)$ whose support is contained in $U$, \cite{collianderkenig}. We will determine the suitable spaces for $k$ and $\ell$ later, which turn out to be $L^2$ based Sobolev-type  spaces,  see \eqref{normHS} and \eqref{normHW}.
In addition, for  $s>\frac12$ we have  the   compatibility condition for the $L^2$ traces: $f|_{y=0}=k|_{t=0}$ and similarly for the wave part. The compatibility condition is necessary since the solutions we are interested in have  continuous $L^2_x$ traces for $s>\frac12$, see Lemma~\ref{lem:embedding}.  
In addition, we will work with the Klein-Gordon equation instead of the wave equation as in \cite{gtv}.

For the Euclidean space, well--posedness theory for the Zakharov system was completed in a series of papers. The final results are sharp due to a scaling argument that appears in \cite{gtv}. In \cite{SS}, existence results for smooth solutions in dimensions $d \leq 3$ were obtained. The regularity assumptions and dimension restrictions were weakened in \cite{AA, SW, OT, KPV}. Later, Bourgain and Colliander, \cite{boco}, proved local well-posedness in all dimensions using the restricted norm method of Bourgain \cite{bourgain}. In \cite{gtv}, Ginibre, Tsutsumi, and Velo applied Bourgain's   method to extend the well--posedness results in all dimensions, covering the full subcritical regularity range for $d\geq 4$. In dimension $d=1$, they obtained local existence at the critical regularity $L^2 \times H^{-\frac12} \times H^{-\frac32}$. In \cite{Hol}, local ill-posedness results were obtained for some regularities outside the well-posedness regime  established in \cite{gtv}. In dimensions two and three, the local well-posedness was obtained in the critical space $L^2 \times H^{-\frac12} \times H^{-\frac32}$ in \cite{BHHT} and \cite{BH} respectively. These results are sharp in the sense that the data-to-solution map fails to be analytic at lower regularity levels.

In this paper, we prove well-posedness for the 2d Zakharov system on the half plane with initial and boundary data in  $L^2$ based Sobolev spaces and obtain low regularity strong solutions. We also impose a non-homogenous boundary constraint at $y=0$.  Wellposedness of \eqref{ZS}  means local existence, uniqueness and continuity with respect to the initial data of distributional solutions. The main part of our paper is the application of the restricted norm method for both free and boundary solution operators which generate correction terms that are absent on $\Bbb R^2$. Moreover certain nonlinear estimates require additional decompositions since the resonant sets in two dimensions are quite large. Here we utilize the machinery developed in \cite{BHHT}, where the main estimates were achieved for
functions which are dyadically localized in frequency and modulation. In some cases and depending on the frequency interactions, one has to further decompose using angular decompositions. For the definition of the  usual Sobolev spaces  and their adapted generalization we refer the reader to the Notation subsection below. 

One has to be careful defining the right notion of strong solutions. More precisely we have the following definition: 

\begin{definition}\label{def:lwp}  
We say \eqref{ZS} is locally well--posed in $H^{s_1}(U)\times H^{s_2}(U)\times H^{s_{2}-1}(U)$, if \\
i) for any $(f,g,h) \in H^{s_1}(U)\times H^{s_2}(U)\times H^{s_{2}-1}(U)$ and $(k,l) \in \cH^{s_1}_{x,t}(U)\times \cH^{s_2}_{x,t}(U)$, with the compatibility conditions $f(x,0)=k(x,0)$, $g(x,0)=l(x,0)$ a.e. for $s>\frac12$, the equation has a distributional solution on $[0,T]$ 
 such that
$$
u \in   X^{s_1,b}_S\cap C^0_tH^{s_1}_{x,y}  \cap C^0_y\cH^{s_1}_{x,t} ,
$$
$$n\in   X^{s_2,b}_W\cap C^0_tH^{s_2}_{x,y}   \cap C^0_y\cH^{s_2}_{x,t} ,$$
where $T=T(\|f\|_{H^{s_1}(U)},\|k\|_{\cH^{s_1}_{x,t} (U) }, \|g\|_{H^{s_2}(U)},  \|h\|_{H^{s_2-1}(U)}, \|l\|_{\cH^{s_2}_{x,t} (U) })$, \\ 
ii) if $f_j\to f$ in $H^{s_1}(U)$, $g_j\to g$ in $H^{s_2}(U)$, $h_j\to h$ in $H^{s_2-1}(U)$, $k_j\to k$ in $\cH^{s_1}_{x,t} (U)$ and $l_j\to l$ in $\cH^{s_2}_{x,t} (U)$, then $(u_j,n_j)\to (u,n)$ in the space above.

\end{definition}

Our first theorem establishes   local well--posedness.

\begin{theorem} \label{thm:local} Fix   $s_2\in (-\frac12, \frac32)\setminus\{\frac12\}$ and let $s_1\neq \frac12$ satisfy   $$\max\big(\frac18,\frac{s_2}2+\frac14,s_2,2s_2-1\big)<s_1<s_2+1.$$ Then, the equation \eqref{ZS} is locally well--posed in $H^{s_1}(U) \times H^{s_2}(U) \times H^{s_2-1}(U)$ in the sense of Definition~\ref{def:lwp}. 
 \end{theorem}
\begin{rmk}
We note that  our result  does not cover the full range of indices $(s_1,s_2)$ that was established for the problem on $\R^2$  in \cite{gtv,BHHT}. One reason is that because of the boundary operator we have to take $b\leq\frac12$ in our $X^{s,b}$ spaces, see propositions~\ref{prop:SXsb} and \ref{prop:KGXsb}.. Another is that the
Kato smoothing estimates for the Duhamel term requires a correction for low as well as high regularities, see propositions~\ref{prop:DKS} and \ref{prop:DKW}.  
\end{rmk}

In addition we obtain the following smoothing estimate:
\begin{theorem} \label{thm:smooth} Fix $s_1, s_2$ and initial and boundary data as in the statement of Theorem~\ref{thm:local} and Definition~\ref{def:lwp}. For any   $\alpha<\min(  \frac12,\frac12+s_2,1-s_1+s_2,\frac52-s_2)$ and $\beta<\min( 2s_1-s_2-\frac12,s_1-s_2,\frac{s_1+1}2-s_2,\frac32-s_2)$, the solution  of \eqref{ZS}  satisfies  
$$
u(x,t)-S_0^t(f,k)(x)\in C^0_tH^{s_1+\alpha}_x([0,T]\times U),
$$
and 
$$
n(x,t)-W_0^t(g,h,l)(x)\in C^0_tH^{s_2+\beta}_x([0,T]\times U),$$
where $T$ is the local existence time, and  $ S_0^t(f,k)$ and $ W_0^t(g,h,l)$ are the solution of the corresponding linear equations, see \eqref{def:V1t} and \eqref{def:S0t}. 
\end{theorem}

To study the half--plane problem we utilize the restricted norm method of Bourgain  \cite{bourgain}. This continues our work initiated in  \cite{etnls}, \cite{etza} and \cite{EGT_KP}, of  establishing the regularity properties of nonlinear dispersive partial differential equations (PDE) on a half line  using the tools that are available in the case of the whole  line. We thus extend the data to the whole plane and use Laplace transform methods to   set up an equivalent integral equation (on $\R^2\times \R $) of the solution, see \eqref{eq:duhamel} below.  We then analyze the integral equation  using  the restricted norm method as in \cite{collianderkenig,etnls,etza} and multilinear $L^2$ convolution estimates. 
In both the standard nonlinear and boundary terms we need to refine our analysis due to resonances that arise from orthogonality issues. After the extension to the full plane, one can utilize the refined analysis in \cite{BHHT} to establish the well-posedness.  Notice that the upper bounds on regularity is an artifact that depends on the method of first extending the data and then restricting to obtain the solution.
Our result is the first low regularity well--posedness result on the half--plane for the 2d Zakharov.  Concerning uniqueness, the solution we obtain for the integral equation \eqref{eq:duhamel}  is unique. However, we cannot obtain a unique strong solution of the original PDE since our solution is a fixed point of \eqref{eq:duhamel}  that depends on the particular extension we use.  

We now discuss briefly the organization of the paper.  In Section 2,  we  introduce the appropriate function spaces that accommodate the solutions. The choice of the right spaces is crucial to close the argument at any level of   regularity.
We also construct the solutions of the linear problem and set up the Duhamel formula  for the full equation. The Duhamel formula incorporates the extension of the data on $\mathbb R^2$ and the evaluation of certain operators at the zero boundary.  
In Section 3, we obtain the a priori linear estimates that we need in order to put our solutions to the right function spaces. Section 4 is the main part of our paper. Here taking advantage by either directly using or combining the results in \cite{BHHT} we prove the nonlinear estimates that imply existence of solutions. This section also provides the tools needed for the proof of Theorem~\ref{thm:smooth}.  In Section 5, we briefly outline the well--known process of establishing LWP and smoothing using the linear and nonlinear estimates of Sections 3 and 4. The last section, Section 6, is an Appendix where we state and if necessary prove a series of analysis lemmas that we use throughout the paper. We finish the introduction with a notation subsection.
 
\subsection{Notation}

Recall that for $s\in \R$, $H^s(\R^d)$ is the closure of $C_0^\infty(\R^d)$ under the norm 
$$
\|f\|_{H^s}=\|f\|_{H^s(\R^d)}:=\Big(\int_{\R^d} \la \zeta\ra^{2s} |\widehat{f}(\zeta)|^2 d\zeta \Big)^{1/2},
$$
where $\la \zeta\ra:=(1+|\zeta|^2)^{1/2}$ and 
$$\widehat{f}(\zeta)=\cF f(\zeta)=\int_{\R^d} f(x)e^{-ix\cdot\zeta} dx$$ 
is the Fourier transform of $f$. We also set the notation
$$
f(\widehat{\zeta_j})=\int_{\R} f(x)e^{-ix_j\zeta_j}dx_j
$$
for the Fourier transform in the $j$th space coordinate.  

For a space-time function $f$, we set the notation
$$
D_0f(x,t)=f(x,0,t).
$$

For $s> -\frac12$, we define the space  $H^s(U)$   as
$$
 H^s(U) := \big\{  g\in \mathcal D(U): \exists  \tilde g\in H^s(\R^2) \text{ so that } \tilde g \chi_U=g \big\},
$$
with the norm 
$$
\|g\|_{H^s(U)}:=\inf\big\{\|\tilde g\|_{H^s(\R^2)}:  \tilde g \chi_U=g \big\}.
$$
The restriction for $s>-\frac12$ is needed since multiplication with $\chi_U$ is not well-defined for $s\leq -\frac12$. 
We say $\tilde g$ is an $H^s(\R^2)$ extension of $g\in H^s(U)$ if $ \tilde g\chi_U=g $   and $\|\tilde g\|_{H^s(\R^2)}\leq 2 \|g\|_{H^s(U)}$.  
Note that, if $g  \in  H^s(U)$ for some $s>\frac12$, then by trace lemma  any $H^s$ extension  is in $C^0_yL^2_x$, and hence $g(x,0)$ is well defined as an $L^2$ function.

Finally, we use $\la x,y\ra $ to denote $\la (x,y)\ra=\sqrt{1+x^2+y^2}$, and we reserve  the symbol $\kappa $ for a smooth compactly supported function of time which is equal to $1$ on $[-1,1]$.

\section{Notion of a solution} 
We start by presenting solution formulas for the linear initial-boundary value system. 
Applying Laplace and Fourier transform as in \cite{bonaetal,RSY,EGT_KP}, we solve 
linear Schr{\oo}dinger  \eqref{eq:sch_lin}  and linear Klein-Gordon \eqref{eq:wave_lin}   when $f=g=h=0$; see \eqref{def:S0t} and \eqref{def:W0t}, respectively. We skip the details of this straightforward calculation.
\begin{multline}\label{def:S0t}
S_0^t(0,k)= \int_\R \int_{\beta+\eta^2<0}  e^{ix\eta+i\beta t} e^{i  y\sqrt{|\beta+\eta^2|}}\widehat {\chi_{t>0}k}(\eta,\beta) d\beta  d\eta
\\+\int_\R \int_{\beta+\eta^2>0}  e^{ix\eta+i\beta t} e^{-y\sqrt{\beta+\eta^2 }}\widehat {\chi_{t>0}k}(\eta,\beta) d\beta d\eta.
\end{multline}
\begin{multline}\label{def:W0t}
W_0^t(0,\ell)=\int_\R \int_{ \eta^2+1<\beta^2}  e^{ix\eta+i\beta t} e^{i \sgn(\beta) y\sqrt{\beta^2 -\eta^2-1}}\widehat {\chi_{t>0}\ell}(\eta,\beta) d\beta d\eta
\\
+\int_\R \int_{ \eta^2+1>\beta^2}  e^{ix\eta+i\beta t} e^{-y\sqrt{ \eta^2+1-\beta^2}}\widehat {\chi_{t>0}\ell}(\eta,\beta) d\beta d\eta.
\end{multline} 
Note that the first summands in the formulas above are well-defined for all $y\in \R$. To make the second summands  well-defined for $y<0$, we replace the exponential function, $e^{-s}$, in 
the formulas with $\rho(s)=e^{-s}\widetilde \chi_{[-1,\infty)}(s)$, where $\widetilde\chi_{[-1,\infty)}$ is a smooth funtion supported on $[-1,\infty)$ which is equal to 1 on $[-\tfrac12,\infty)$. Note that $\rho$ is a Schwartz function.  We also introduce suitable changes of variables to write 
$S_0^t=S_1+S_2$, $W_0^t=W_1+W_2 $, where 
$$
S_1(k)=2 \int_\R \int_{0}^\infty  e^{i(x,y)\cdot \xi } e^{  -i t |\xi|^2}    \widehat {\chi_{t>0}k}(\xi_1,-|\xi|^2) \  \xi_2 \  d\xi_2d\xi_1  , 
$$ 
$$S_2(k)=2\int_\R \int_{0}^\infty  e^{ix\xi_1}e^{i t (\xi_2^2-\xi_1^2)} \rho(y\xi_2) \widehat {\chi_{t>0}k}(\xi_1,\xi_2^2-\xi_1^2) \  \xi_2 \ d\xi_2d\xi_1 ,
$$ 
$$
W_1(\ell)= \int_{\R^2}   e^{i(x,y)\cdot \xi }  e^{ i t \sgn(\xi_2)\la \xi\ra   }  \widehat {\chi_{t>0}\ell}(\xi_1,\sgn(\xi_2)\la \xi\ra) \frac{|\xi_2| \  d\xi }{\la \xi\ra},
$$
$$
W_2(\ell)=\int_\R \int_{ \eta^2+1>\beta^2 }  e^{ix\eta+i\beta t} \rho(y\sqrt{ \eta^2+1-\beta^2})\widehat {\chi_{t>0}\ell}(\eta,\beta) d\beta d\eta
$$
$$=\int_\R \int_{1+\xi_1^2 >\xi_2^2}  e^{ix\xi_1}e^{i t \sgn(\xi_2)\sqrt{ 1+\xi_1^2-\xi_2^2} } \rho(y|\xi_2|)  \widehat {\chi_{t>0}\ell}\big(\xi_1,\sgn(\xi_2)\sqrt{ 1+\xi_1^2 -\xi_2^2}\big) \frac{|\xi_2| \  d\xi }{\sqrt{1+ \xi_1^2 -\xi_2^2}}.
$$ 
 
We will now define  $S_0^t$ and $W_0^t$ that  solve  the linear Schr\"odinger \eqref{eq:sch_lin} and Klein-Gordon \eqref{eq:wave_lin} with extended non-homogenous  initial  and boundary data (the restriction to the domain is independent of the extension). 
Let $g_{e }$ and $h_{e }$ be extensions of $g$ and $h$ to  $\R^2$  with the property that 
$$
\|g_{e  }\|_{H^{s_2}(\R^2)}\les \|g\|_{H^{s_2}(U)}, \,\,\,\,\,\|h_{e }\|_{ H^{s_2-1}(\R^2)}\les \|h\|_{ H^{s_2-1}(U)}.
$$
We define
$$
\psi_{\pm}(x)=g_{e }(x)\pm i \la \nabla \ra^{-1}  h_{e }(x) \in H^{s_2}(\R^2).
$$
Let $W_0^t(\psi^{\pm},\ell)$ be defined on $\R^2_{x,y}\times \R_t$  by 
\be\label{def:V1t}
W_0^t(\psi^{\pm},\ell)(x,y)  =  \frac12  \left[e^{  -it \la \nabla \ra} \psi_{+}+ e^{ i t \la \nabla \ra} \psi_{-}\right](x,y) +  W_0^t(0, \ell-r)(x,y).
\ee
Here 
\be\label{def:r}
r(x,t)= \frac12 \kappa(t)\left(D_0 e^{    -it \la \nabla \ra} \psi_{+} + D_0 e^{  it \la \nabla \ra} \psi_{-}\right), 
\ee
where $D_0$ is restriction to the set $y=0$.

Note that  the restriction of $W_0^t(\psi^{\pm}, \ell) $ to  $U \times [0,1] $ solves:
\begin{equation}\label{eq:wave_lin}
\left\{
\begin{array}{l} 
n_{tt}+(1-\Delta) n=0,\quad (x,y) \in U,\, t>0,\\ 
n(x,y,0)=g(x,y),\quad n_t(x,y,0)=h(x,y), \quad  n(x ,0,t)=\ell(x ,t).
\end{array}
\right.
\end{equation}
We define $S_0^t(f_e,k)$ similarly, where $f_e$ is an extension of $f$ to $\R^2$. 
\be\label{def:S0t}
S_0^t(f_e,k)=e^{it\Delta}f_e + S_0^t(0, k-D_0e^{it\Delta}f_e ).
\ee
Note that  the restriction of $S_0^t(f_e,k) $ to  $U \times [0,1] $ solves:
\begin{equation}\label{eq:sch_lin}
\left\{
\begin{array}{l} 
iu_t + \Delta u=0,\quad (x,y) \in U,\, t>0,\\ 
u(x,y,0)=f(x,y),\quad    u(x ,0,t)=k(x ,t).
\end{array}
\right.
\end{equation}

In order for the solutions $S_0^t$, $W_0^t$   above to make sense we require $\chi_{t>0}k\in \cH_S^{s_1} $ and $\chi_{t>0}\ell \in \cH_W^{s_2} $, where   
\be\label{normHS}
\mathcal H^{s_1}_S=\{k:\R^2_{x,t}\to \C: \widehat { k}\big(\xi_1,\pm \xi_2^2-\xi_1^2\big) \, \xi_2 \la  \xi\ra^{s_1} \in L^2_{ \xi} \}
\ee
\be\label{normHW}
\mathcal H^{s_2}_W=\{\ell:\R^2_{x,t}\to\C: \widehat {\ell}\big(\xi_1,\xi_2) \,  |\xi_2^2-\xi_1^2-1|^{\frac14} \la \xi\ra^{s_2-\frac12} \in L^2_{ \xi} \}.
\ee
We can rewrite the second definition as  follows by considering the regions $\xi_2^2>\xi_1^2+1$ and $\xi_2^2\leq \xi_1^2+1$ separately:
\begin{multline}\label{normHW1}
\mathcal H^{s_2}_W=\{\ell:\R^2_{x,t}\to\C: \widehat {\ell}\big(\xi_1,\sgn(\xi_2) \la \xi\ra\big) \,  \xi_2 \la \xi\ra^{s_2-1} \in L^2_{ \xi} \} 
\\ \cap  \{\ell:\R^2\to\C: \widehat {\ell}\big(\xi_1,\sgn(\xi_2) \sqrt{ 1+\xi_1^2-\xi_2^2 }\big) \, \frac {\xi_2 \la  \xi\ra^{s_2-\frac12}}{(1+ \xi_1^2 -\xi_2^2 )^{\frac14}} \chi_{\xi_2^2<\xi_1^2+1 } \in L^2_{ \xi} \}.
\end{multline}
  
 The norms $\mathcal H^{s_1}_S(U)$   and $\mathcal H^{s_2}_W(U)$ are defined analogously.
In the next section we will prove   Kato smoothing estimates, see 
Proposition~\ref{prop:KSS} and Proposition~\ref{prop:KSKG}, that implies that these spaces are  the natural choice
for the boundary data. In particular, we will conclude that  $S_0^t(f_e,k)\in C^0_t H^{s_1}(U)$ and $W_0^t(\psi^{\pm},\ell)\in C^0_t  H^{s_2}(U)$.

 We now establish embedding and extension properties of these spaces.       
\begin{lemma}\label{lem:embedding}
For $s>1$, the space $\mathcal H^s_{S}$ and $\mathcal H^s_{W}$  embed  continuously into $C^0_{x,t}$. Moreover, for $s>\frac12$, we have the trace lemma; the spaces $\mathcal H^s_{S}$ and $\mathcal H^s_{W}$ embed  continuously into $C^0_tL^2_x$, and  in particular $\sup_t \|\varphi\|_{ L^2_x}\les \| \varphi\|_{\mathcal H^s_{S}}$, $\sup_t \|\varphi\|_{ L^2_x}\les \| \varphi\|_{\mathcal H^s_{W}}$.
\end{lemma}
\begin{proof} This is similar to Lemma~2.1 from \cite{EGT_KP}. We provide the proof only for the wave part,  $\mathcal H^s_{W}$.
 
For the first claim, it suffices to prove that given $\varphi \in \mathcal H^s_W$,
$\|\widehat\varphi\|_{L^1(\R^2)}\les \| \varphi\|_{\mathcal H^s_W}$.
By the Cauchy-Schwarz inequality this follows from
$$
\int_{\R^2}\frac{1}{|\xi_2^2-\xi_1^2-1|^{\frac12} \la \xi\ra^{2s -1}}\ d\xi  \les 1,
$$
which holds for  $s>1$ by considering the regions $|\xi_2| \leq 2\la\xi_1\ra$, $|\xi_2|> 2\la\xi_1\ra$  separately.  

Similarly, for the second claim it suffices to prove that
$\|\widehat\varphi\|_{L^2_{\xi_1} L^1_{\xi_2} }\les \| \varphi\|_{\mathcal H^s_W}$. By the Cauchy-Schwarz inequality in the $\xi_2$ integral this follows from
$$
\sup_{\xi_1} \int_{\R} \frac{1}{|\xi_2^2-\xi_1^2-1|^{\frac12} \la \xi\ra^{2s -1}}d\xi_2 \les 1,
$$
which holds for $s>\frac12$ by the same decomposition of the regions.
\end{proof}
Next, we prove that under usual regularity conditions, multiplication with a characteristic function in time is a continuous operator on these spaces.   
\begin{lemma}\label{lem:extend}
For $-\frac32<s<\frac12$, we have 
$$
\|\chi_{t>0} \varphi(x,t)\|_{\mathcal H^s_{S} }\les \| \varphi \|_{\mathcal H^s_S (U)}.  
$$ 
Moreover, for $\frac12<s<\frac52$, we have the same bound  provided that the trace $\varphi(x,0)$ is zero.

Analogously, for $-\frac12<s<\frac12$, we have 
$$ \|\chi_{t>0} \varphi(x,t)\|_{\mathcal H^s_{W} }\les \| \varphi \|_{\mathcal H^s_W (U)}.
$$ 
Moreover, for $\frac12<s<\frac32$, we have the same bound  provided that the trace $\varphi(x,0)$ is zero.

\end{lemma}
\begin{proof} This is similar to Lemma~2.2 from \cite{EGT_KP}. We provide the proof only for the wave part,  $\mathcal H^s_{W}$.
Since $\mF\big(\chi_{t>0}  \varphi\big)(\xi_1,\xi_2) = H \widehat\varphi(\xi_1,\xi_2) ,$
where $H$ is essentially the Hilbert transform in the $\xi_2$ variable:
$$ Hf(\xi_1,\xi_2)= \mF_t\big(\chi_{t>0} f(\xi_1,t^\vee)\big)(\xi_2),$$ 
It suffices to prove that 
$$m(\xi)=  |\xi_2^2-\xi_1^2-1|^{\frac12} \la \xi\ra^{2s -1}\sim ||\xi_2| -\la \xi_1\ra|^{\frac12}  (|\xi_2| +\la \xi_1\ra)^{2s -\frac12}$$ is an $A_2$ weight in $\xi_2$ uniformly in $\xi_1$ for $-\frac12<s<\frac12$, see \cite{JD}. Recalling that the $A_2$ constant is invariant under dilations  and scaling, we can replace $m$ with $||\xi_2| -1|^{\frac12}  (|\xi_2| +1)^{2s -\frac12}$. 

Since $m$ is even, it suffices to consider intervals $[a,b]$: $0<a<b<\infty$. When $0<a<b\ll 1$, we have $m\sim 1$, which is an $A_2$ weight. 
Similarly, when $1\ll a<b$, we have $m\sim \la \xi_2\ra^{2s}$, which is an $A_2$ weight for $-\frac12<s<\frac12$. When $a,b\sim 1$, we have $m\sim |1- \xi_2|^{\frac12}$, which is also an $A_2$ weight. It remains to consider the case $0<a\ll1\ll b$. We estimate
$$
\int_a^b m \les b^{2s+1}\sim (b-a)^{2s+1},\,\,\,\, \int_a^b m^{-1} \les b^{-2s+1}\sim (b-a)^{-2s+1} 
$$
 for $-\frac12<s<\frac12$. Therefore, $m$ is an $A_2$ weight. 
 
For the second part, we note that 
$\|\chi_{t>0}  \varphi \|_{\mathcal H^s_W(\R^2)} =\|T\big(\chi_{t>0}  \varphi \big)\|_{\mathcal H^{s-1}_{W}(\R^2)}$, where $T$ is the multiplier operator with the multiplier $\la \xi\ra $. Furthermore
$$
\|T\big(\chi_{t>0}   \varphi\big)\|_{\mathcal H^{s-1}_{W}(\R^2)}\les \|T_1\big(\chi_{t>0}   \varphi\big)\|_{\mathcal H^{s-1}_{W}(\R^2)} + \|T_2\big(\chi_{t>0} \varphi\big)\|_{\mathcal H^{s-1}_{W}(\R^2)},
$$
where $T_1$ and $T_2$ are multiplier operators with multipliers $1+|\xi_1|$ and $\xi_2$, respectively. Note that $T_1$ commutes with multiplication by  $\chi_{t>0}$. Since $\varphi$ has trace zero, $\partial_t\big(\chi_{t>0}  \varphi\big) =\chi_{t>0}   \partial_t  \varphi$ in the sense of distributions. Therefore, we also have   $T_2\big(\chi_{t>0}   \varphi\big)= \chi_{t>0}  T_2\varphi$. Also using the first part for $s-1$, we obtain
 $$
\|T\big(\chi_{t>0}   \varphi\big)\|_{\mathcal H^{s-1}_{W}(\R^2)}\les \|T_1  \varphi \|_{\mathcal H^{s-1}_{W}(U)} + \|T_2   \varphi \|_{\mathcal H^{s-1}_{W}(U)} \les \|\varphi \|_{\mathcal H^{s}_{W}(U)}.
$$
\end{proof}

Now we write a system of integral equations equivalent to  \eqref{ZS} on $[0,T]$, $T<1$:
\begin{equation}\label{eq:duhamel} \left\{
\begin{array}{l}
u(t)= \kappa(t) S_0^t\big(f_e, k \big) -i \kappa(t/T) \int_0^t e^{i(t- t^\prime)\Delta}   F(u,n) \,d t^\prime  +i\kappa(t) S_0^t\big(0,  q  \big), \\
n(t)= \kappa(t) W_0^t\big(\psi^{\pm}, \ell \big) + \frac12\kappa(t/T) (n_+ +  n_-) -\frac12 \kappa(t) W_0^t(0,z), 
\end{array}\right.
\end{equation}
where
\be \label{def:F}
F(u,n)=\kappa(t ) n u,\,\,\,\, \text{ and }\,\, q(t)=  \kappa(t ) D_0\Big(\int_0^t e^{i(t-t^\prime)\Delta}  F(u,n)\, d t^\prime \Big).
\ee
\be\label{def:npm}
n_\pm =   \pm  \int_0^t e^{\mp i (t-t^\prime)\la \nabla \ra} G(u )dt^\prime,\,\,\,\,G(u ):=-\kappa(t )   \frac{\Delta}{\la \nabla \ra}  |u|^2-\kappa(t )   \frac{1}{\la \nabla \ra} n, 
\ee
\be\label{def:z}
z(t)=  \kappa(t) D_0(n_++n_-) .
\ee

To establish the local well-posedness of \eqref{eq:duhamel}, we work with the following Bourgain spaces on $\R^2\times \R$
$$
\|u\|_{X_S^{s_1,b}}=\big\| \widehat u(\xi ,\tau) \la \tau + |\xi|^2  \ra^b \la \xi \ra^{s_1}  \big\|_{L^2_{\xi, \tau}}=\|e^{-it\Delta}u\|_{H^{s_1}_{x,y}H^b_t},
$$
$$
\|n\|_{X_W^{s_2,b}}=\inf_{n=n_+ + n_-} \left(\|n_+\|_{X_{W,+}^{s_2,b}}+ \|n_-\|_{X_{W,-}^{s_2,b}} \right),
$$
where
$$
\|n\|_{X_{W,\pm}^{s_2,b}}=\big\| \widehat n(\xi, \tau) \big\la \tau \pm \la\xi\ra   \big\ra^b \la \xi \ra^{s_2}  \big\|_{L^2_{\xi,\tau}}=\|e^{\pm it\la \nabla\ra }n\|_{H^{s_2}_{x,y}H^b_t}
$$

 We  recall the embedding $X^{s,b}\subset C^0_t H^{s} $ for $b>\frac{1}{2}$ and the following inequalities from \cite{bourgain,gtv,etbook}. 

For any $s,b$ we have
\begin{equation}\label{eq:xs1}
\|\eta(t)W_{\R^2} g\|_{X^{s,b}}\les \|g\|_{H^s}.
\end{equation}
For any $s\in \mathbb R$,  $0\leq b_1<\frac12$, and $0\leq b_2\leq 1-b_1$, we have
\begin{equation}\label{eq:xs2}
\Big\| \eta(t) \int_0^t W_\R(t-t^\prime)  F(t^\prime ) dt^\prime \Big\|_{X^{s,b_2} }\lesssim   \|F\|_{X^{s,-b_1} }.
\end{equation}
Moreover, for $T<1$, and $-\frac12<b_1<b_2<\frac12$, we have
\begin{equation}\label{eq:xs3}
\|\eta(t/T) F \|_{X^{s,b_1}}\les T^{b_2-b_1} \|F\|_{X^{s,b_2}}.
\end{equation}
 
\section{A priori linear estimates}\label{sec:lin}

In this section we establish a priori linear estimates that are crucial for the wellposedness theory. In particular, we obtain   Kato smoothing  estimates for the linear groups and Duhamel integrals, and $X^{s,b}$ space estimates for the boundary operators.
\begin{prop}\label{prop:SXsb} (Schr\"odinger $X^{s,b}$ Estimate) For $s \geq 0$ and $ b \leq  \frac12$, we have 
$$
\|\kappa(t) S_j(k)\|_{X_S^{s  ,b}}\les \|\chi_{t>0}k\|_{\mathcal H^{s  }_S},\quad j=1,2.
$$
\end{prop}
\begin{proof} For $S_1$ the claim is immediate from the definition of $\mathcal{H}_S^{s }$ (using the ``-'' sign in the norm) since
\[
\mathcal F(\kappa S_1(k))(\xi, \tau)=2 \xi_2 \chi_{\xi_2 > 0} \widehat{\kappa}(\tau + |\xi|^2) \widehat {\chi_{t>0}k}(\xi_1, -|\xi|^2) . 
\]
 For $S_2$, we calculate
\begin{align*}
\mathcal F(\kappa S_2(k))(\xi ,\tau)
& =2\int_0^\infty \widehat{\kappa}(\tau+\xi_1^2-\eta^2) \widehat{\rho}(\tfrac{\xi_2}{\eta})\widehat {\chi_{t>0}k}(\xi_1,\eta^2-\xi_1^2) d\eta\\
&= \int_0^\infty \widehat{\kappa}(\tau+\xi_1^2-\eta^2) \widehat{\rho}(\tfrac{\xi_2}{\eta}) \frac{1}{\eta \la \xi_1,\eta \ra^{s   }  }K(\xi_1,\eta) d\eta,
\end{align*}
where
$$
\|K\|_{L^2 }\les \|\chi_{t>0}k\|_{\mathcal H^{s   }_S}.
$$
Therefore, the $X_S^{s , b}$ norm is
\[
\Big\| \la \xi \ra ^{s }
\int_0^\infty \widehat{\kappa}(\tau+\xi_1^2-\eta^2) \la \tau + |\xi|^2  \ra^{b}  \widehat{\rho}(\tfrac{\xi_2}{\eta})\frac{1}{\eta \la \xi_1,\eta \ra^{s  } }K(\xi_1,\eta) d\eta\Big\|_{L^2_{\xi,\tau}}.
\]
Since $\kappa,\rho$ are Schwartz functions, we may bound this, for any constant $M > 0$, by
\be\label{kappaK}
\left\| \la \xi \ra ^{s }
\int_0^\infty \frac{\la \tau + |\xi|^2  \ra^{b}}{\la \tau+\xi_1^2-\eta^2\ra^{100}} \  \frac{|\eta|^{M-1}}{\eta^M+\xi_2^M}\  \frac{|K(\xi_1,\eta)|}{\la \xi_1,\eta \ra^{s   }}\ d\eta
\right\|_{L^2_{\xi, \tau}}. 
\ee

When $\eta^2+\xi_2^2<1$, we have 
$\la \xi  \ra^{s } 
\lesssim 
\la  \xi_1,\eta \ra^{s   }$
 and
$\la \tau + |\xi|^2 \ra 
\sim 
\la \tau+\xi_1^2-\eta^2\ra$. 
Taking $M=2$ and moving the $L^2_{\xi_1,\tau}$ norm inside the integral in \eqref{kappaK}, we obtain the bound
\begin{multline}\label{magic}
  \Big\|
\int_\R    \frac{|\eta|}{ \eta^2+\xi_2^2  }   \|K(\xi_1,\eta)\|_{L^2_{\xi_1}} d\eta\Big\|_{L^2_{ \xi_2  }}
 =\Big\|
\int_\R    \frac{|r|}{ r^2+1  }   \|K(\xi_1,r \xi_2)\|_{L^2_{\xi_1}} dr\Big\|_{L^2_{ \xi_2  }}
\\
 \leq \int_\R    \frac{|r|}{ r^2+1  }   \|K(\xi_1,r \xi_2)\|_{L^2_{\xi }}  dr \\
  \leq \int_\R    \frac{|r|^{\frac12}}{ r^2+1  }   \|K( \xi)\|_{L^2_{\xi }} dr\les\|K\|_{L^2 }\les \|\chi_{t>0}k\|_{\mathcal H^{s   }_S}.
\end{multline}

When $\eta^2+\xi_2^2>1$,   we consider first the case that $|\xi_2| \lesssim |\xi_1|$ or $|\xi_2| \lesssim |\eta|$. In this case we have $\la \xi  \ra^{s_1}   \lesssim \la \xi_1,\eta \ra ^ {s_1  }$. Again taking $M=2$ in \eqref{kappaK}, and then integrating in $\xi_2$, we obtain the bounds
\begin{multline}\label{Keta2b}
 \les \Big\|
\int_0^\infty \frac{1}{\la \tau-\eta^2+\xi_1^2\ra^{100-b}}  \frac{|\eta|}{(\eta^2+\xi_2^2)^{1-b}}   |K(\xi_1,\eta)| d\eta\Big\|_{L^2_{\xi ,\tau}} \\
 \\ \leq  \Big\|
\int_0^\infty \frac{|\eta|^{2b -  \frac12}}{\la \tau-\eta^2+\xi_1^2\ra^{10}}      |K(\xi_1,\eta)|  d\eta\Big\|_{L^2_{\xi_1, \tau}} .
\end{multline}
For the part of the integral where $\eta \in [0,1]$, bounding this by $\| K \|_{L^2 }$ is trivial. For $\eta > 1$, since $b\leq \frac12$, we have the bound
\[ \Big\|
\int_1^\infty \frac{|\eta|^{ \frac12}}{\la \tau-\eta^2+\xi_1^2\ra^{10}}    |K(\xi_1,\eta)| d\eta\Big\|_{L^2_{\xi_1, \tau}}. 
\]
Using Lemma~\ref{lem:chvar} we bound this by $\|K\|_{L^2 }$, as desired. 

It remains to address the case when $\eta^2 + \xi_2^2 > 1$ and $|\xi_2| \gg |\eta|, |\xi|$. 
In this case, we see that \eqref{kappaK} can be bounded by 
\begin{align*}
\left\|
\int_0^\infty \frac{1}{\la \tau-\eta^2+\xi_1^2\ra^{100 - b}}  \frac{|\eta|^{M-1}}{|\xi_2|^{M - s  - 2b}}  \frac{|K(\xi_1,\eta)|}{\la \xi_1,\eta\ra^{s  }} d\eta
\right\|_{L^2_{\xi ,\tau}}. 
\end{align*}
Taking the $L^2_{\xi_2}$ norm, keeping in mind that $|\xi_2| \gg |\eta|, |\xi|$, we obtain
\begin{align*}
\left\|
\int_0^\infty \frac{1}{\la \tau-\eta^2+\xi_1^2\ra^{10 }}  \frac{|\eta|^{M-1}}{\la \xi_1,\eta \ra^{M - s  - 2b - \frac12}}  \frac{|K(\xi_1,\eta)|}{\la  \xi_1,\eta \ra^{s   }} d\eta
\right\|_{L^2_{\xi_1,\tau}} \\
\lesssim
\left\|
\int_0^\infty \frac{1}{\la \tau-\eta^2+\xi_1^2\ra^{10 }}  \frac{|\eta|^{M-1}}{\la \eta \ra^{M - 2b - \frac12  }} |K(\xi_1,\eta)| d\eta
\right\|_{L^2_{\xi_1,\tau}}. 
\end{align*}
This is bounded by $\| K\|_{L^2_{\eta, \xi}}$ by the bound for \eqref{Keta2b} above.  
\end{proof}

\begin{prop}\label{prop:KGXsb} (Klein-Gordon $X^{s,b}$ Estimate) For any $s  \geq -\frac12$ and $b\leq \frac12$, we have 
$$
\|\kappa(t) W_j(\ell)\|_{X_W^{s  ,b}}\les \|\chi_{t>0}\ell\|_{\mathcal H^{s  }_W},\quad j=1,2.
$$
\end{prop}
\begin{proof} Note that it suffices to consider $b = \frac12$.

For $W_1$ the claim is immediate from the definitions of $W_1$ and the first part of the norm $\mathcal{H}_W^{s }$ in \eqref{normHW} and noting that the $X_W^{s  ,b}$ is defined as an infemum.

For $W_2$, we have 
\begin{multline*}
\F(\kappa W_2(\ell))(\xi, \tau) = \\ \int_{1+\xi_1^2  >\eta^2}   \widehat\kappa\big(\tau-   \sgn(\eta)\sqrt{1+ \xi_1^2-\eta^2} \big)  \widehat{\rho}(\xi_2/|\eta|)  \widehat {\chi_{t>0}\ell}\big(\xi_1,\sgn(\eta)\sqrt{ 1+\xi_1^2 -\eta^2}\big) \frac{  d\eta  }{\sqrt{1+\xi_1^2-\eta^2}}.
\end{multline*}
Letting 
$$
L(\xi_1,\eta):= \la \xi_1  \ra ^{s  -\frac12 } \big|  \widehat {\chi_{t>0}\ell}\big(\xi_1,\sgn(\eta) \sqrt{ 1+ \xi_1^2-\eta^2 }\big) \big|\, \frac {|\eta| }{(1+ \xi_1^2-\eta^2)^{\frac14}} \chi_{\eta^2<\xi_1^2+1 },
$$ 
note that $\|L\|_{L^2_{\xi_1,\eta}}\les \|\chi_{t>0}\ell\|_{\mathcal H^{s  }_W}$. We have  the following bound for $X^{s ,b}_{W,\pm}$ norm
$$
\Big\|  \int_{1+\xi_1^2 >\eta^2}   \frac{\la \xi  \ra^{s }}{ \la \xi_1  \ra ^{s  -\frac12 }}\left|\widehat\kappa(\tau-   \sgn(\eta)\sqrt{ 1+ \xi_1^2-\eta^2} )\right| \big\la \tau \pm  \la \xi \ra  \big\ra^{b}  \frac{ |\widehat{\rho}(\xi_2/|\eta|)| }{|\eta| (1+ \xi_1^2-\eta^2)^{\frac14}}   L(\xi_1,\eta) d\eta\Big\|_{L^2_{\xi ,\tau}}.
$$
Since the $X^{s , b}_{W}$ norm is defined as an infimum, and since $\kappa,\rho$ are Schwartz functions, it suffices to consider 
\be\label{wavexsbtemp}
\Big\|  \int_{1+\xi_1^2 >\eta^2} \frac{\la \xi  \ra^{s }}{ \la \xi_1  \ra ^{s  -\frac12 }} 
\frac{\la \tau -  \la \xi \ra   \ra^{b} }{\la\tau-  \sqrt{ 1+ \xi_1^2-\eta^2} \ra^{100} }\frac{|\eta|^{M-1} }{ (\eta^M+\xi_2^M) (1+ \xi_1^2-\eta^2)^{\frac14}}    L(\xi_1,\eta) d\eta\Big\|_{L^2_{\xi, \tau}}.
\ee
When $\eta^2+\xi_2^2\ll 1$, observe that $ \la \tau -  \la \xi \ra    \ra\sim \big\la \tau- \sqrt{ 1+ \xi_1^2-\eta^2}\big\ra$,  $\la \xi   \ra \sim \la \xi_1 \ra$, and $ (1+ \xi_1^2-\eta^2)^{\frac14}\sim \la \xi_1\ra^{\frac12}$. Taking $L^2_{\xi_1,\tau}$ norm inside and choosing $M=2$, we estimate the norm by 
$$
\Big\|
\int_{\eta^2+\xi_2^2\ll 1}   \frac{|\eta|}{ \eta^2+\xi_2^2  }   \|L(\xi_1,\eta)\|_{L^2_{\xi_1}} d\eta\Big\|_{L^2_{\xi_2  }},
$$
which can be bounded as in the proof of the previous theorem, see \eqref{magic}.

When $\eta^2+\xi_2^2\gtrsim 1$ and $|\xi_2| \gg |\xi_1| $, we bound \eqref{wavexsbtemp} by 
$$
\Big\|  \int_{\xi_1^2 +1>\eta^2} \frac{1}{ \la \xi_1  \ra ^{s -\frac12  }} 
\frac{1}{\la\tau-  \sqrt{1+ \xi_1^2-\eta^2} \ra^{100-b} }\frac{|\eta|^{M-1} }{ \la \xi_2 \ra^{M - s  - b}(1+ \xi_1^2-\eta^2)^{\frac14}}    L(\xi_1,\eta) d\eta\Big\|_{L^2_{\xi, \tau}}.
$$
Taking the $L^2_{\xi_2}$ norm yields
$$
\Big\|  \int_{\xi_1^2 +1>\eta^2} \frac{1}{ \la \xi_1   \ra ^{M - 1-b}} 
\frac{|\eta|^{M-1}}{\la\tau-  \sqrt{ 1+ \xi_1^2-\eta^2} \ra^{10 } (1+ \xi_1^2-\eta^2)^{\frac14}}  L(\xi_1,\eta) d\eta\Big\|_{L^2_{\xi_1,\tau}}.
$$
This is bounded by $\| L\|_{L^2_{\xi_1, \eta}}$ for $b\leq \frac12$ and $M\geq \frac32$ using Lemma~\ref{lem:chvar}.

When $\eta^2+ \xi_2^2 \gtrsim 1$ and $|\xi_2| \lesssim  |\xi_1| $,   the ratio of $s $ multipliers is $\sim 1$, for any $s $. 
We also have 
\begin{equation} \label{eq:taylorapprox} \frac{\big\la \tau -  \la\xi\ra   \big\ra^{b}}{\big\la \tau- \sqrt{ 1+ \xi_1^2-\eta^2}\big\ra^{100}} \les \frac{1+ \frac{(\eta^2+\xi_2^2)^{b}}{\la \xi\ra ^{b}}      }{\big\la \tau- \sqrt{  1+ \xi_1^2-\eta^2}\big\ra^{100-b}}. 
\end{equation}

We first consider the contribution of the second term  in the numerator of \eqref{eq:taylorapprox}. Taking $M=2$, it is bounded by  
$$
\Big\|  
\int_{ \eta^2< 1+\xi_1^2  } \frac1{\big\la \tau- \sqrt{ 1+ \xi_1^2-\eta^2}\big\ra^{10 }}  \frac{ |\eta|\la \xi_1\ra^{\frac12} }{(\eta^2+\xi_2^2 )^{1 - b} \la \xi\ra^b (1+ \xi_1^2-\eta^2)^{\frac14}  }  |L(\xi_1,\eta)| d\eta\Big\|_{L^2_{\xi, \tau}}.
$$
Taking the $L^2_{\xi_2}$ norm inside the integral and recalling that   $b = \frac12$, we have
$$
\Big\|\frac{ \eta }{(\eta^2+\xi_2^2)^{\frac12}   }\Big\|_{L^2_{\xi_2}}\les |\eta|^{ \frac12} ,
$$
we obtain the bound
$$
\Big\|
\int_{ \eta^2<1+ \xi_1^2  } \frac{|\eta|^{ \frac12}}{\big\la \tau- \sqrt{ 1+ \xi_1^2-\eta^2}\big\ra^{10 }(1+ \xi_1^2-\eta^2)^{\frac14}} |L(\xi_1,\eta)| d\eta\Big\|_{L^2_{\xi_1, \tau}}.
$$
This is bounded by $\| L\|_{L^2_{\xi_1, \eta}}$  using Lemma~\ref{lem:chvar}.

Finally, we consider the contribution of $1$ in the numerator of \eqref{eq:taylorapprox}:
$$
\Big\|  
\int_{ \eta^2< 1+\xi_1^2  } \frac1{\big\la \tau- \sqrt{ 1+ \xi_1^2-\eta^2}\big\ra^{10 }}  \frac{ |\eta|\la \xi_1\ra^{\frac12} }{(\eta^2+\xi_2^2 )   (1+ \xi_1^2-\eta^2)^{\frac14}  }  |L(\xi_1,\eta)| d\eta\Big\|_{L^2_{\xi, \tau}}.
$$
This can be bounded as in \eqref{magic} after taking the $L^2_{\xi_1,\tau}$ norm inside in the case $|\eta|\ll \la \xi_1\ra$. In the case $|\eta|\gtrsim \la \xi_1\ra$, we obtain the bound by  taking the $L^2_{\xi_2}$ norm inside and then applying Lemma~\ref{lem:chvar}.
\end{proof}

\begin{prop}\label{prop:KSS} (Kato smoothing for Schr\"odinger) For $s \geq 0$, and $f\in H^{s }(\R^2)$, we have $\kappa(t)  e^{it\Delta}f \in C^0_y  \H^{s }_S$, and 
$$
\big\|\kappa(t)  e^{it\Delta}f  \big\|_{L^\infty_y \H^{s }_S}\les \|f\|_{H^{s }(\R^2)}.
$$  
\end{prop}
\begin{proof} 
We set $U(x,y,t) =\kappa(t)  e^{it\Delta}f (x,y,t)$. Taking the Fourier transform in $x$ we get
$$
 U(\widehat{\xi_1},y,t)=\int_{\R} \kappa(t) \ e^{-i |\xi|^2  t} e^{i\xi_2 y}\widehat{f}(\xi ) \ d\xi_2.
$$
Now the Fourier transform in $t$ gives
$$
U(\widehat{\xi_1},y,\widehat{\eta})  =\int_{\R} \widehat{\kappa}(\eta+|\xi|^2  )\ e^{i\xi_2  y} \widehat{f}(\xi ) \ d\xi_2 .
$$ 
By dominated convergence theorem, the statement follows from the claim:
$$
I:=\int_{\R^2} \la \xi_1,\eta\ra^{2s}  \eta^2  \Big(\int_{\R} |\widehat{\kappa}(\xi_2^2\pm\eta^2)| | \widehat{f}(\xi )|d\xi_2\Big)^2 d\eta d\xi_1 \les \|f\|_{H^s}^2.
$$
Applying the Cauchy-Schwarz inequality to the $\xi_2$-integral we get
$$
\Big(\int_{\R} |\widehat{\kappa}(\xi_2^2\pm\eta^2)| |\widehat{f}(\xi )|d\xi_2\Big)^2 \leq \left\|\widehat{\kappa}(\xi_2^2\pm\eta^2)\right\|_{L^1_{\xi_2}} \int_{\R}|\widehat{\kappa}(\xi_2^2\pm\eta^2)| |\widehat{f}(\xi )|^2d\xi_2,
$$
and using the fact that for any $M>0$
$$
|\widehat{\kappa}( \xi_2^2\pm\eta^2 )|\lesssim \frac{1}{\big\la  \xi_2^2\pm\eta^2 \big\ra^{M}},
$$
we have the estimate
$$
\left\|\widehat{\kappa}(\xi_2^2\pm\eta^2)\right\|_{L^1_{\xi_2}} \les \frac1{\la \eta\ra }.
 $$ We now combine these estimates to bound the integral $I$ as follows
$$
I\lesssim \int_{\R^3} \la \xi_1,\eta\ra^{2s} \frac{|\eta|  }{\big\la  \xi_2^2\pm\eta^2  \big\ra^{M}} |\widehat{f }(\xi )|^2 d\eta  d\xi.
$$
It suffices to consider only `-'   in the denominator and show that 
$$
J:=\int_{\R} \la \xi_1,\eta\ra^{2s} \frac{|\eta|  }{\big\la  \xi_2^2-\eta^2  \big\ra^{M}} d\eta \lesssim \la \xi  \ra^{2s}.
$$
This is trivial when   $|\eta|\les |\xi_2|$. 

When  $|\eta|>2  |\xi_2|$, using $|\xi_2^2-\eta^2|\gtrsim \eta^2$ we have
$$
J\lesssim \int  \la \xi_1,\eta\ra^{2s} \frac{|\eta| d\eta}{(1+\eta^2)^M}\lesssim \Big(\sup_{|\eta|>2  |\xi_2|} \frac{ \la \xi_1,\eta\ra^{2s}}{(1+\eta^2)^{M/2}}\Big)\,\int_{\R} \frac{|\eta|d\eta}{(1+\eta^2)^{M/2}}   \lesssim \la \xi  \ra^{2s},
$$ 
for $M>\max(2,2s)$.  
\end{proof}
\begin{prop}\label{prop:KSKG} (Kato smoothing for Klein-Gordon) For $s \in\R$, and $f \in H^{s }(\R^2)$, we have $\kappa(t)  e^{\pm it\la \nabla\ra}f \in C^0_y  \H^{s }_W$, and 
$$
\big\|\kappa(t)  e^{\pm it\la \nabla\ra  }f  \big\|_{L^\infty_y\H^{s }_W}\les \|f\|_{H^{s }(\R^2)}.
$$  
\end{prop}
\begin{proof}
Setting $W(x,y,t) =\kappa(t)  [e^{\pm it\la \nabla\ra } f] (x,y)$, and taking the Fourier transform in $x$ and $t$ we get
$$
W(\widehat{\xi_1},y,\widehat{\eta})  =\int_{\R} \widehat{\kappa}\big(\eta\mp \la \xi\ra \big) e^{i\xi_2 y} \widehat{f}(\xi )d\xi_2.
$$ 
By dominated convergence theorem, the statement follows from the claim:
\be\label{eq:temp1}
 \int_{\R^2}    \la\xi_1,\eta\ra^{2s-1} |\eta^2-\xi_1^2-1|^{\frac12}  \Big(\int_{\R} \big|\widehat{\kappa}\big( \eta  \pm\la  \xi\ra \big)\big| \frac{|\widehat{f}(\xi )|}{\la \xi\ra^s}d\xi_2\Big)^2 d\eta d\xi_1\les \|f\|_{L^2}^2.
\ee
In the cases $|\eta|\ll\la \xi\ra$ and $|\eta|\gg\la \xi\ra$, the bound is easy due to the Schwartz decay of $\widehat{\kappa}$. When $|\eta|\sim \la \xi\ra$, we bound the kernel by 
$$
K_{\xi_1}(\xi_2,\eta):=\chi_{|\eta|\sim\la \xi\ra} \frac{||\eta|-\la\xi_1\ra|^{\frac14}}{\la \xi\ra^{\frac14} \la |\eta|-\la \xi\ra\ra^M},
$$
and consider the subcases: \\
i) $|\xi_2|\gtrsim \la \xi_1\ra$, \\
ii) $|\xi_2|\ll \la \xi_1\ra$ and $||\eta|-\la \xi_1\ra|\les \frac{|\xi_2|^2}{\la \xi_1\ra}$, \\
iii) $|\xi_2|\ll \la \xi_1\ra$ and $||\eta|-\la \xi_1\ra|\gg \frac{|\xi_2|^2}{\la \xi_1\ra}$.

In case i), we have 
$$\sup_{\xi_1,\xi_2} \int K_{\xi_1}(\xi_2,\eta)d\eta \les 1, \text{ and } \sup_{\xi_1,\eta} \int K_{\xi_1}(\xi_2,\eta)d\xi_2 \les 1,$$
which suffices. The second inequality follows from the change of variable $\rho=\la \xi\ra$, noting that $d\rho=\frac{|\xi_2|}{\la\xi\ra}d\xi_2\sim d\xi_2.$

In case ii), the kernel is bounded by $|\xi_2|^{\frac12}\la \xi_1\ra^{-\frac12}\la |\eta|-\la\xi\ra\ra^{-M}$. We have 
$$
\int K_{\xi_1}(\xi_2,\eta)d\eta \les |\xi_2|^{\frac12}\la \xi_1\ra^{-\frac12}, \text{ and } 
\int K_{\xi_1}(\xi_2,\eta)|\xi_2|^{\frac12}\la \xi_1\ra^{-\frac12} d\xi_2 \les 1,$$
by the same change of variable in the second integral.

In case iii), the kernel is bounded by 
$ \chi_{|\xi_2|\ll (|\eta|-\la \xi_1\ra|)^{1/2} \la\xi_1\ra^{1/2}} \la \xi_1\ra^{-\frac14}\la |\eta|-\la\xi_1\ra\ra^{-M+\frac14}$. Therefore, 
$$
\int K_{\xi_1}(\xi_2,\eta)d\eta \les \la \xi_1\ra^{-\frac14}, \text{ and } 
$$
$$\int K_{\xi_1}(\xi_2,\eta) d\xi_2 \les \la \xi_1\ra^{-\frac14}\la |\eta|-\la\xi_1\ra\ra^{-M+\frac14}(|\eta|-\la \xi_1\ra|)^{1/2} \la\xi_1\ra^{1/2}\les \la \xi_1\ra^{ \frac14},$$
which suffices. 
 \end{proof}

\begin{prop} \label{prop:tracelem}(Trace lemmas for the boundary operators) 
Fix $s_1, s_2\geq -\frac12$ and assume that $\rho$ is mean zero and it's first moment is zero.  For $\chi_{t>0}k\in \mathcal{H}^{s_1}_S$, we have   $S_0^t(0,k)\in C^0_{t}H^{s_1}_{x,y} \cap C^0_y \mathcal{H}^{s_1}_S$. Similarly,  for  $\chi_{t>0}\ell \in  \mathcal{H}^{s_2}_W$, we have  $W_0^t(0,\ell)\in C^0_{t} H^{s_2}_{x,y} \cap C^0_y \mathcal{H}^{s_2}_W$.  
\end{prop}
\begin{proof} We first consider the Schr\"odinger part.  
We start with $C^0_{t}H^{s_1}_{x,y}$ assuming that the boundary data is in $\mathcal{H}^{s_1}_S$. The claim is immediate for $S_1$. 
For $S_2$, letting 
$$\widehat g(\xi):=e^{it(\xi_2^2-\xi_1^2)} \xi_2\ \widehat{\chi_{t>0}k}(\xi_1,\xi_2^2-\xi^2_1)$$
and by using the time continuity of Schr\"odinger evolution in $H^s$, it suffices to prove that 
$$
T(g)(x,y)=\int \rho( \xi_2 y) g(x,\widehat{\xi_2}) d\xi_2
$$
is bounded in $H^{s_1}$. We have
$$
\widehat{T(g)}(\theta_1,\theta_2)=\int \widehat\rho( \theta_2/\xi_2 )  \widehat{g}(\theta_1, \xi_2 ) \frac{d\xi_2}{\xi_2}= \int \widehat\rho( \eta )   \widehat{g}(\theta_1, \theta_2/\eta) \frac{ d\eta}{\eta}.
$$
Therefore,  
$$
\|T(g)\|_{H^{s_1}}\leq \int |\widehat\rho( \eta )|  \big\| \la \theta_1,\theta_2\ra^{s_1} \widehat{g}(\theta_1, \theta_2/\eta) \big\|_{L^2_{\theta_1,\theta_2}}\frac{ d\eta}{|\eta|}
$$
Note that 
$$
 \big\| \la \theta_1,\theta_2\ra^{s_1} \widehat{g}(\theta_1, \theta_2/\eta) \big\|_{L^2_{\theta_1,\theta_2}} \les \|g\|_{H^{s_1}}\left\{\begin{array}{ll}\la \eta\ra^{s_1}\eta^{\frac12} ,& s_1\geq 0\\
 |\eta|^{s_1+\frac12}+|\eta|^{\frac12}, & s_1<0.
\end{array}\right.
 $$
 This finishes the proof by the assumptions on $\rho$.  
 
We now prove that $S_0^t(0,k)\in   C^0_y \mathcal{H}^{s_1}_S$. Note  that by \eqref{def:S0t},  we have (uniformly in $y$)
$$
|[S_0^t(0,k)](\widehat{\xi_1},y,\widehat \beta)|\les  |\widehat {\chi_{t>0}k}(\xi_1,\beta)|.
$$
Therefore, by dominated convergence theorem, 
$$\|S_0^t(0,k)\|_{C^0_y \mathcal{H}^{s_1}_S}\les \|\chi_{t>0} k\|_{ \mathcal{H}^{s_1}_S}.
$$ 
The proof of $W_0^t(0,\ell)\in C^0_y \mathcal{H}^{s_2}_W$ is analogous.

It remains to see that  $W_0^t(0,\ell)\in C^0_{t} H^{s_2}_{x,y} $. Once again, for $W_1$, the proof is easy.   For $W_2$,
$$
W_2(\ell)=\int_\R \int_{ \eta^2+1>\beta^2 }  e^{ix\eta+i\beta t} \rho(y\sqrt{ \eta^2+1-\beta^2})\widehat {\chi_{t>0}\ell}(\eta,\beta) d\beta d\eta
$$
Noting that the effect of each $x$ or $y$ derivative is bounded by $\la \eta\ra$ on the Fourier side, it suffices to consider the case $-\frac12\leq s_2\leq 0$.  We write $W_2(\ell)$ as 
$$
\int_\R \int_{|\beta|<\la\eta\ra/2}   e^{ix\eta+i\beta t} \rho(y\sqrt{ \eta^2+1-\beta^2})\widehat {\chi_{t>0}\ell}(\eta,\beta) d\beta d\eta$$
$$
+\int_\R \int_{\la\eta\ra/2<|\beta|<\la\eta\ra }   e^{ix\eta+i\beta t} \rho(y\sqrt{ \eta^2+1-\beta^2})\widehat {\chi_{t>0}\ell}(\eta,\beta) d\beta d\eta.
$$
For the first summand, taking Fourier transform, it suffices to prove that
$$
\Big\| \la \eta,\theta\ra^{s_2}\int_{|\beta|<\la\eta\ra/2}   \frac{| \widehat\rho(\theta/\sqrt{ \eta^2+1-\beta^2})|}{\sqrt{ \eta^2+1-\beta^2}} |\widehat {\chi_{t>0}\ell}(\eta,\beta)| d\beta \Big\|_{L^2_{\eta,\theta}}\les \|\widehat {\chi_{t>0}\ell}(\eta,\beta) \la\eta\ra^{s_2} \|_{L^2_{\eta,\beta}}
$$
Since $s_2\leq 0$, we have $\la \eta,\theta\ra^{s_2}\leq \la \eta\ra^{s_2}$. Taking $L^2_\theta$ norm inside, we bound the left hand side by
$$
\Big\| \int_{|\beta|<\la\eta\ra/2}   \frac{1}{\sqrt{ \la \eta \ra}} |\widehat {\chi_{t>0}\ell}(\eta,\beta)| \la \eta \ra^{s_2} d\beta \Big\|_{L^2_{\eta }}.
$$
The claim follows by Cauchy-Schwarz inequality in $\beta$ integral. 

We rewrite the second summand as
$$ \int_\R \int_{|\xi_2|\leq \frac{\sqrt3}{2} \la \xi_1\ra}  e^{ix\xi_1}e^{i t \sgn(\xi_2)\sqrt{ 1+\xi_1^2-\xi_2^2} } \rho(y|\xi_2|)  \widehat {\chi_{t>0}\ell}\big(\xi_1,\sgn(\xi_2)\sqrt{ 1+\xi_1^2 -\xi_2^2}\big) \frac{|\xi_2|   d\xi }{\sqrt{1+ \xi_1^2 -\xi_2^2}}.
$$ 
Note that this is $T(g)(x,y)$ with  
$$\widehat g(\xi):=e^{i t \sgn(\xi_2)\sqrt{ 1+\xi_1^2-\xi_2^2} }\frac{|\xi_2|   }{\sqrt{1+ \xi_1^2 -\xi_2^2}} \widehat {\chi_{t>0}\ell}\big(\xi_1,\sgn(\xi_2)\sqrt{ 1+\xi_1^2 -\xi_2^2}\big)\chi_{|\xi_2|\leq \frac{\sqrt3}{2} \la \xi_1\ra}.$$
The claim follows from the  $H^{s_2}$ boundedness of $T$ that we proved above by dominated convergence theorem after noting that $H^{s_2}$ norm of $g$ is bounded by $ \mathcal{H}^{s_2}_W$ norm of $\chi_{t>0}\ell $.
\end{proof}

We now establish Kato smoothing estimates for the Duhamel integrals. 
 \begin{prop}\label{prop:DKS}(Kato smoothing for Schr\"odinger Duhamel)  For any $ b<\frac12$, we have
\be
\Big\|\kappa  \int_0^te^{ i(t- t^\prime)\Delta} G  dt^\prime  \Big\|_{C^0_y\H^{s_1}_S}\les   \left\{ \begin{array}{ll} \|G\|_{X_S^{s_1,-b}}& \text{ for } 0\leq s_1 \leq \frac12,  \\
\|G\|_{X_S^{s_1,-b}}+  \|G\|_{X_S^{\frac12+,\frac{s_1}2-\frac34}}   & \text{ for }  \frac12 < s_1. \end{array}
\right.
\ee
\end{prop}
\begin{proof} By dominated convergence theorem, it suffices to consider the evaluation at $y=0$.  
$$
 \F_{x\to \xi} \Big(\kappa(t) D_{y=0}\Big(\int_0^t e^{ i(t- t^\prime)\Delta} G  dt^\prime  \Big)\Big)=   \kappa(t)  \int_{\R^2} \frac{e^{it \lambda }-e^{- it  (\xi^2+\theta^2) }}{i(\lambda+ \xi^2+\theta^2  )}   \widehat G(\xi, \theta,\lambda) d\theta d\lambda 
$$
Therefore
$$
 \F_{x\to \xi, t\to\beta} \Big(\kappa(t) D_{y=0}\Big(\int_0^t e^{ i(t- t^\prime)\Delta} G  dt^\prime  \Big)\Big)=    \int_{\R^2} \frac{\widehat\kappa(\beta- \lambda ) -\widehat\kappa(\beta+\xi^2+\theta^2 ) }{i(\lambda+\xi^2+\theta^2  )}   \widehat G(\xi, \theta,\lambda) d\theta d\lambda 
$$

Therefore the  norm is bounded by 
$$
\Big\| \eta \la\eta,\xi\ra^{s_1}  \int_{\R^2} \frac{\widehat\kappa(\pm\eta^2-\xi^2- \lambda ) -\widehat\kappa(\pm\eta^2 +\theta^2 ) }{ \lambda+  \xi^2+\theta^2 }   \widehat G(\xi, \theta,\lambda) d\theta d\lambda\Big\|_{L^2_{\xi,\eta}  }. 
$$
We first consider the case $|\lambda+  \xi^2+\theta^2|<1$, by the mean value theorem we estimate the norm in this case by 
\begin{multline*}
\Big\| \eta \la\eta,\xi\ra^{s_1}  \int_{\R^2} \chi_{|\lambda+  \xi^2+\theta^2|<1}\frac{1}{\la\eta^2 -\theta^2\ra^M }  | \widehat G(\xi, \theta,\lambda) |  d\theta  d\lambda \Big\|_{L^2_{\xi,\eta}  }\les \\  \Big\| \eta \la\eta,\xi\ra^{s_1}  \int_{\R }  \frac{1}{\la\eta^2 -\theta^2\ra^M }  \| \widehat G(\xi, \theta,\lambda)\|_{L^2_\lambda} d\theta  \Big\|_{L^2_{\xi,\eta}  }
\end{multline*}
In the last inequality we used Cauchy-Schwarz inequality in $\lambda $ integral. 
As above it suffices to prove that the operator with kernel
$$
K_\xi(\eta,\theta)=\frac{ |\eta| \la\eta,\xi\ra^{s_1}}{\la\theta,\xi\ra^{s_1}}   \frac{1}{\la\eta^2 -\theta^2\ra^M } 
$$ 
is bounded $L^2_\theta $ to $L^2_\eta$ uniformly in $\xi$ for sufficiently large $M$. 
In the case $|\eta|\les |\theta|$ both $\eta$ and $\theta$ integrals of the kernel are $\les 1$ using 
$$
\int \frac{1}{\la\eta^2 -\theta^2\ra^M } d\theta\les \la \eta\ra^{-1}.
$$  
In the case $|\eta|\gg|\theta|$, the kernel is bounded by $|\eta|^{-M/2} $ and the same bounds hold. 

In the case $|\lambda+  \xi^2+\theta^2|>1$, we bound the norm by
$$
\Big\| \eta \la\eta,\xi\ra^{s_1}  \int_{\R^2} \frac{\la \pm\eta^2-\xi^2- \lambda \ra^{-M} +\la\pm\eta^2 +\theta^2 \ra^{-M} }{\la  \lambda+  \xi^2+\theta^2 \ra}   |\widehat G(\xi, \theta,\lambda)| d\theta d\lambda\Big\|_{L^2_{\xi,\eta}  }. 
$$
In the case $\la\eta,\xi\ra\les \la\theta,\xi\ra$, letting 
$$L:=\la\theta,\xi \ra^{s_1}\la \lambda+\xi^2+\theta^2 \ra^{-b} |\widehat G(\xi,\theta,\lambda)|,$$
and noting that $\|L\|_{L^2}\les \|G\|_{X^{s_1,-b}}$, it suffices to prove that 
the operator with kernel 
$$
K_\xi(\eta,(\theta,\lambda)) =\eta    \frac{\la \pm\eta^2-\xi^2- \lambda \ra^{-M} +\la\pm\eta^2 +\theta^2 \ra^{-M} }{\la  \lambda+  \xi^2+\theta^2 \ra^{1-b} }   
$$
is bounded $L^2_{\theta,\lambda}\to L^2_\eta$ uniformly in $\xi$. The $\eta$ integral is bounded by $\la  \lambda+  \xi^2+\theta^2 \ra^{-1+b}$ by the change of variable $\rho=\eta^2$. By the weighted Schur test the following bound  suffices:  
$$
\int_{\R^2} |\eta|    \frac{\la \pm\eta^2-\xi^2- \lambda \ra^{-M} +\la\pm\eta^2 +\theta^2 \ra^{-M} }{\la  \lambda+  \xi^2+\theta^2 \ra^{2-2b} } d \theta d\lambda    \les |\eta| \int_\R\frac1{\la\theta^2-\eta^2\ra^{2-2b}} d\theta \les 1. 
$$
Last inequality follows by considering the cases $|\theta|\ll|\eta|$ and $|\theta|\gtrsim|\eta|$.

In the case $\la \eta,\xi\ra \gg \la \theta,\xi\ra$, which implies $|\eta| \gg \la \theta,\xi\ra$, we consider the cases $0\leq s_1\leq \frac12$ and $s_1>\frac12$ seperately.   In the former case we have the following kernel
$$
K_\xi(\eta,(\theta,\lambda)) =\frac{|\eta |^{1+s_1}}{\la\theta,\xi\ra^{s_1}}    \frac{\la \pm\eta^2-\xi^2- \lambda \ra^{-M} +|\eta|^{-2M}  }{\la  \lambda+  \xi^2+\theta^2 \ra^{1-b} } .  
$$
We omit the contribution of $|\eta|^{-2M}$ which can be handled as above. For the remaining part note that
$$
\int K_\xi(\eta,(\theta,\lambda)) |\eta|^{-s_1} d\eta \les \la\theta,\xi\ra^{-s_1}
\la  \lambda+  \xi^2+\theta^2 \ra^{-1+b}.
$$
Therefore, the following bound suffices
$$
\int_{\R^2} K_\xi(\eta,(\theta,\lambda))\la\theta,\xi\ra^{-s_1}
\la  \lambda+  \xi^2+\theta^2 \ra^{-1+b}  d\theta  d\lambda \les  |\eta|^{-s_1}.
$$ 
Note that after the $\lambda$ integral  we have the bound
$$
\int_{|\theta|\ll |\eta| } \frac{|\eta|^{1+s_1}}{ \la\theta,\xi\ra^{2s_1}
|\eta|^{4-4b}}  d\theta \les |\eta|^{1-2s_1+} |\eta|^{4b-3+s_1}\les |\eta|^{-s_1}
$$
for $b<\frac12$.

In the latter case, it suffices to prove that  the operator with kernel
$$
K_\xi(\eta,(\theta,\lambda)) =\frac{|\eta |^{1+s_1}}{\la\theta,\xi \ra^{\frac12+}}    \frac{\la \pm\eta^2-\xi^2- \lambda \ra^{-M} +|\eta|^{-2M}  }{\la  \lambda+  \xi^2+\theta^2 \ra^{\frac14+\frac{s_1}2 } }   
$$
is bounded $L^2_{\theta,\lambda}\to L^2_\eta$ uniformly in $\xi$. We again ignore the contribution of $|\eta|^{-2M}$. 
For the remaining part note that
$$
\int K_\xi(\eta,(\theta,\lambda)) |\eta|^{-s_1} d\eta \les \frac{1}{\la\theta,\xi \ra^{\frac12+}
\la  \lambda+  \xi^2+\theta^2 \ra^{\frac14+\frac{s_1}2 }}.
$$
Therefore, the following bound suffices
$$
\int_{\R^2} \frac{K_\xi(\eta,(\theta,\lambda)) }{\la\theta,\xi \ra^{\frac12+}
\la  \lambda+  \xi^2+\theta^2 \ra^{\frac14+\frac{s_1}2 }}  d\theta  d\lambda \les  |\eta|^{-s_1}.
$$ 
Note that after the $\lambda$ integral  we have the bound
$$
\int_{|\theta|\ll |\eta| } \frac{|\eta|^{1+s_1}}{ \la\theta \ra^{1+}
|\eta|^{1+2s_1}}  d\theta \les |\eta|^{ - s_1 }.
$$ 
\end{proof}

\begin{prop}\label{prop:DKW} (Kato smoothing for Klein-Gordon Duhamel) We will take   $b<\frac12$.
\[
\left\| \kappa(t) \int_0^t e^{\pm i (t - t') \la\nabla\ra }  G d t' \right\|_{C^0_y \mathcal{H}^{s}_W} 
\lesssim  \left\{ \begin{array}{ll} \| G \|_{X^{s , -b}_W} + \Big\|  \frac{\chi_{|\lambda|\ll \la \xi\ra}\la \lambda \ra^{s} |\widehat G (\xi,\lambda)|}{  \la \xi\ra^{\frac12-}  }  \Big\|_{L^2_{\lambda,\xi }} & \text{ for }s  <0,  \\
\| G \|_{X^{s , -b}_W}  & \text{ for } 0\leq s  \leq \frac12,  \\
\| G \|_{X^{s, -b}_W} + \|G\|_{X_W^{\frac12+,s-1}}  & \text{ for }  \frac12 < s. \end{array}
\right. \]
\end{prop}

\begin{proof} By dominated convergence theorem, it suffices to consider the evaluation at $y=0$.  
$$
 \F_{x\to \xi_1} \Big(\kappa(t) D_{y=0}\Big(\int_0^t e^{ i(t- t^\prime)\la\nabla\ra  } G  dt^\prime  \Big)\Big)=   \kappa(t)  \int_{\R^2} \frac{e^{it \lambda }-e^{ it  \la\xi\ra  }}{i(\lambda-  \la \xi\ra  )}   \widehat G(\xi,\lambda) d\xi_2 d\lambda 
$$
Therefore
$$
 \F_{x\to \xi_1, t\to\beta} \Big(\kappa(t) D_{y=0}\Big(\int_0^t e^{i (t- t^\prime)\la\nabla\ra} G  dt^\prime  \Big)\Big)=    \int_{\R^2} \frac{\widehat\kappa(\beta- \lambda ) -\widehat\kappa(\beta- \la\xi\ra) }{i(\lambda-  \la\xi\ra )}   \widehat G(\xi,  \lambda) d\xi_2 d\lambda. 
$$

The $\mathcal H^{s}_W$ norm is 
\be\label{eq:duhkat}
\Big\|    \la\xi_1 ,\beta\ra^{s-\frac12} |\beta^2-\la \xi_1\ra^2|^{\frac14}  \int_{\R^2} \frac{\widehat\kappa(\beta- \lambda ) -\widehat\kappa(\beta- \la \xi\ra) }{i(\lambda-  \la \xi\ra)}   \widehat G(\xi,  \lambda) d\xi_2 d\lambda \Big\|_{L^2_{\xi_1,\beta}}.
\ee
Considering the cases $|\lambda-\la\xi\ra|<1$ and $|\lambda-\la\xi\ra|\geq1$ separately, we have the bound
$$
\frac{|\widehat\kappa(\beta- \lambda ) -\widehat\kappa(\beta- \la \xi\ra) |}{|\lambda-  \la \xi\ra|}\les \frac{\la \beta- \lambda \ra^{-M} +\la \beta- \la \xi\ra\ra^{-M} }{\la \lambda-  \la \xi\ra\ra}.
$$
We start with the contribution of  the term $\la \beta- \la \xi\ra\ra^{-M}$ in the numerator above. 
Letting 
\be\label{Lxi}
L(\xi,\lambda):=\la\xi  \ra^{s}\la \lambda-\la \xi\ra \ra^{-b} |\widehat G(\xi, \lambda)|,
\ee
and noting that $\|L\|_{L^2_{\xi,\lambda}}\les \|G\|_{X^{s,-b}}$, it suffices to prove that 
\be\label{eq:duhkat1}
\Big\|    \la\xi_1 ,\beta\ra^{s-\frac12} |\beta^2-\la \xi_1\ra^2|^{\frac14}  \int_{\R^2}\frac{ \la \beta- \la \xi\ra\ra^{-M} }{\la\xi  \ra^{ s} \la \lambda-  \la \xi\ra\ra^{1-b}} L(\xi,  \lambda) d\xi_2 d\lambda \Big\|_{L^2_{\xi_1,\beta}}\les \|L\|_{L^2_{\xi,\lambda}}.
\ee
After a Cauchy-Schwarz inequality in $\lambda$ integral, we obtain 
$$
\Big\|    \la\xi_1 ,\beta\ra^{s-\frac12} |\beta^2-\la \xi_1\ra^2|^{\frac14}  \int_{\R}\frac{ 1}{\la\xi  \ra^{ s}\la |\beta|- \la \xi\ra\ra^{M}  } \|L(\xi,  \lambda)\|_{L^2_{\lambda}} d\xi_2  \Big\|_{L^2_{\xi_1,\beta}}.
$$
This can be bounded using the estimate for \eqref{eq:temp1} above.

We now consider the  contribution of the   term $\la \beta- \lambda\ra^{-M}$ in the numerator above. The correction terms on the right hand side of the assertion of the proposition arise in this case when $|\beta|\not \sim \la \xi\ra$. In the region $|\beta|\sim \la \xi\ra$, we use the function $L$ defined in \eqref{Lxi} and bound the contribution by
\begin{multline}\label{eq:duhkat2}
\Big\|    \la\xi_1 ,\beta\ra^{s-\frac12} |\beta^2-\la \xi_1\ra^2|^{\frac14}  \int_{\R^2}\chi_{|\beta|\sim\la \xi\ra}\frac{ \la \beta- \lambda\ra^{-M} }{\la\xi  \ra^{ s} \la \lambda-  \la \xi\ra\ra^{1-b}} L(\xi,  \lambda) d\xi_2 d\lambda \Big\|_{L^2_{\xi_1,\beta}}\\
\les \Big\|    \int_{\R^2}K_{\xi_1}(\beta,(\xi_2,\lambda))
 L(\xi,  \lambda) d\xi_2 d\lambda \Big\|_{L^2_{\xi_1,\beta}}, 
\end{multline}
where  
$$
K_{\xi_1}(\beta,(\xi_2,\lambda)) =\chi_{|\beta|\sim\la \xi\ra} \frac{   ||\beta|-\la \xi_1\ra|^{\frac14}}{\la\xi  \ra^{ \frac14} \la \lambda-  \la \xi\ra\ra^{1-b}\la \beta- \lambda \ra^{M }}.
$$
We need to prove that \eqref{eq:duhkat2} $\les  \|L\|_{L^2_{\xi,\lambda}}$.

Consider the subcases: \\
i) $|\xi_2|\gtrsim \la \xi_1\ra$, \\
ii) $|\xi_2|\ll \la \xi_1\ra$ and $||\beta|-\la \xi_1\ra|\les \frac{|\xi_2|^2}{\la \xi_1\ra}$, \\
iii) $|\xi_2|\ll \la \xi_1\ra$ and $||\beta|-\la \xi_1\ra|\gg \frac{|\xi_2|^2}{\la \xi_1\ra}$.

In case i), we have 
$$\int K_{\xi_1}(\beta,(\xi_2,\lambda)) d\beta \les \frac1{\la \lambda-  \la \xi\ra\ra^{1-b}}, \text{ and }$$
$$ \int \frac{K_{\xi_1}(\beta,(\xi_2,\lambda))}{\la \lambda-  \la \xi\ra\ra^{1-b}} d\xi_2d\lambda  \les \int \frac{1}{\la \beta-  \la \xi\ra\ra^{2-2b}} d\xi_2 \les 1,$$
which suffices. The last inequality follows from the change of variable $\rho=\la \xi\ra$, noting that $d\rho=\frac{|\xi_2|}{\la\xi\ra}d\xi_2\sim d\xi_2.$

In case ii), the kernel is bounded by $|\xi_2|^{\frac12}\la \xi_1\ra^{-\frac12}\la \lambda-  \la \xi\ra\ra^{-1+b}\la \beta- \lambda \ra^{-M}$. We have 
$$\int K_{\xi_1}(\beta,(\xi_2,\lambda)) d\beta \les \frac{|\xi_2|^{\frac12}}{\la \xi_1\ra^{\frac12}\la \lambda-  \la \xi\ra\ra^{1-b}}, \text{ and }$$
$$ \int \frac{|\xi_2|^{\frac12}}{\la \xi_1\ra^{\frac12}\la \lambda-  \la \xi\ra\ra^{1-b}}   K_{\xi_1}(\beta,(\xi_2,\lambda))  d\xi_2d\lambda  \les \int \frac{|\xi_2|}{\la \xi_1\ra\la \beta-  \la \xi\ra\ra^{2-2b}} d\xi_2 \les 1,$$
by the same change of variable in the last integral.

In case iii), noting that
$$
\la \beta-\lambda\ra\la\lambda-\la\xi\ra\ra \gtrsim    \la |\beta|-\la\xi\ra\ra=\la |\beta|-\la \xi_1\ra+O(\xi_2^2/\la\xi_1\ra)\ra\sim \la |\beta|-\la \xi_1\ra\ra,
$$
the kernel is bounded by 
$ \chi_{|\xi_2|\ll (|\beta|-\la \xi_1\ra|)^{1/2} \la\xi_1\ra^{1/2}} \la \xi_1\ra^{-\frac14}\la |\beta|-\la\xi_1\ra\ra^{b-\frac34}\la \beta-\lambda\ra^{-M/2}$. Therefore, 
$$ \int  K_{\xi_1}(\beta,(\xi_2,\lambda)) d\lambda d\xi_2 \les \la |\beta|-\la\xi_1\ra\ra^{b-\frac14} \la \xi_1\ra^{ \frac14}, \text{ and } $$
$$\int K_{\xi_1}(\beta,(\xi_2,\lambda)) \la |\beta|-\la\xi_1\ra\ra^{b-\frac14} \la \xi_1\ra^{ \frac14} d\beta \les \int  \la |\beta|-\la\xi_1\ra\ra^{2b-1} \la \beta-\lambda\ra^{-M/2} d\beta\les 1.$$

  Now consider the region $|\beta|\ll \la \xi\ra$. For $s\geq 0$, we use the function $L$  defined in \eqref{Lxi} and bound the contribution by
\begin{multline}\label{eq:duhkat3}
\Big\|    \la\xi_1 ,\beta\ra^{s-\frac12} |\beta^2-\la \xi_1\ra^2|^{\frac14}  \int_{\R^2}\chi_{|\beta|\ll\la \xi\ra}\frac{ \la \beta- \lambda\ra^{-M} }{\la\xi  \ra^{ s} \la \lambda-  \la \xi\ra\ra^{1-b}} L(\xi,  \lambda) d\xi_2 d\lambda \Big\|_{L^2_{\xi_1,\beta}}\\
\les \Big\|    \int_{\R^2}K_{\xi_1}(\beta,(\xi_2,\lambda))
 L(\xi,  \lambda) d\xi_2 d\lambda \Big\|_{L^2_{\xi_1,\beta}}, 
\end{multline}
where  
$$
K_{\xi_1}(\beta,(\xi_2,\lambda)) =\chi_{|\beta|\ll\la \xi\ra} \frac{ 1}{    \la \xi\ra^{1-b}\la \beta- \lambda \ra^{M/2}}.
$$
We need to prove that \eqref{eq:duhkat3} $\les  \|L\|_{L^2_{\xi,\lambda}}$, which follows from 
$$
\int K_{\xi_1}(\beta,(\xi_2,\lambda)) d\beta \les \la \xi\ra^{b-1}, \text{ and } $$
$$\int \la \xi\ra^{b-1} K_{\xi_1}(\beta,(\xi_2,\lambda)) d\xi_2d\lambda \les \int \frac{ d\xi_2d\lambda}{    \la \xi_2\ra^{2-2b}\la \beta- \lambda \ra^{M/2}}\les 1.
$$
 
 For $s< 0$ when $|\beta|\ll \la \xi\ra$, the contribution to \eqref{eq:duhkat} is bounded by 
 $$
 \Big\|    \la\xi_1 ,\beta\ra^{s-\frac12} |\beta^2-\la \xi_1\ra^2|^{\frac14}  \int_{\R^2} \chi_{|\beta|\ll\la \xi\ra}\frac{ \la \lambda-  \beta\ra^{-M}}{\la \lambda-\la\xi\ra\ra}   \widehat G(\xi,  \lambda) d\xi_2 d\lambda \Big\|_{L^2_{\xi_1,\beta}}.
 $$ 
 When $|\lambda|\gtrsim \la \xi\ra$ the kernel is bounded by $\la \lambda\ra^{-M}$, which suffices to bound the contribution by $ \| G \|_{X^{s , -b}_W}$. In the case $|\lambda|\ll \la \xi\ra$, we have the bound 
 $$
 \Big\|     \int_{\R^2} \chi_{ |\lambda|\ll\la \xi\ra}\frac{\la\lambda\ra^{s }  }{ \la \lambda-  \beta\ra^{ M/2 }\la \xi\ra }   \widehat G(\xi,  \lambda) d\xi_2 d\lambda \Big\|_{L^2_{\xi_1,\beta}}.
 $$
 This can be bounded by the correction term in the proposition by Cauchy-Schwarz inequality in $\xi_2$ integral and by  Young's inequality.

 Finally, we consider the  region $|\beta|\gg\la \xi\ra$. For $s\leq \frac12$, we use the function $L$  defined  in \eqref{Lxi}  and consider
\begin{multline}\label{eq:duhkat4}
\Big\|    \la\xi_1 ,\beta\ra^{s-\frac12} |\beta^2-\la \xi_1\ra^2|^{\frac14}  \int_{\R^2}\chi_{|\beta|\gg\la \xi\ra}\frac{ \la \beta- \lambda\ra^{-M} }{\la\xi  \ra^{ s} \la \lambda-  \la \xi\ra\ra^{1-b}} L(\xi,  \lambda) d\xi_2 d\lambda \Big\|_{L^2_{\xi_1,\beta}}\\
\les \Big\|    \int_{\R^2}K_{\xi_1}(\beta,(\xi_2,\lambda))
 L(\xi,  \lambda) d\xi_2 d\lambda \Big\|_{L^2_{\xi_1,\beta}}, 
\end{multline}
where  
$$
K_{\xi_1}(\beta,(\xi_2,\lambda)) =\chi_{|\beta|\gg\la \xi\ra} \frac{ 1}{    \la \xi\ra^{s}\la \beta\ra^{1-b-s} \la \beta- \lambda \ra^{M/2}}.
$$
We need to prove that \eqref{eq:duhkat4} $\les  \|L\|_{L^2_{\xi,\lambda}}$, which follows from 
$$
\int K_{\xi_1}(\beta,(\xi_2,\lambda)) d\beta \les \la \xi\ra^{b-1}, \text{ and } $$
$$\int \la \xi\ra^{b-1} K_{\xi_1}(\beta,(\xi_2,\lambda)) d\xi_2d\lambda \les \int \frac{ d\xi_2d\lambda}{    \la \xi_2\ra^{2-2b}\la \beta- \lambda \ra^{M/2}}\les 1.
$$
 For $s>\frac12$ when $|\beta|\gg \la \xi\ra$, the contribution to \eqref{eq:duhkat} is bounded by 
 $$
 \Big\|    \int_{\R^2} \chi_{|\beta|\gg\la \xi\ra}\frac{\la \beta\ra^{s }  \la \lambda-  \beta\ra^{-M}}{\la \lambda-\la\xi\ra\ra}   \widehat G(\xi,  \lambda) d\xi_2 d\lambda \Big\|_{L^2_{\xi_1,\beta}}\les  \Big\|    \int_{\R^2} \frac{ \la \lambda-  \beta\ra^{-M/2}}{\la \lambda-\la\xi\ra\ra^{1-s}}   \widehat G(\xi,  \lambda) d\xi_2 d\lambda \Big\|_{L^2_{\xi_1,\beta}} .
 $$ This is bounded by the correction term in the statement by the Cauchy-Schwarz inequality in $\xi_2$ integral and Young's inequality.
\end{proof}

\section{Nonlinear Estimates} \label{sec:nonlin}
We present a priori nonlinear eestimates in this section. The proofs rely in part on the bounds established in \cite{BHHT}.  

\begin{prop}\label{prop:Ssmooth} For $s_1\geq  0$, $s_2\geq  -\frac12$, and $  a<\min(\frac12,\frac12+s_2, 1-s_1+s_2)$, we have
\[\| nu \| _{X_S^{s_1 + a, -b}} \lesssim \| n \|_{X_{W,\pm}^{s_2, b}} \|u\|_{X_S^{s_1, b}},\]
provided that $b<\frac12$ is sufficiently close to $\frac 12$. 
\end{prop}

\begin{prop}\label{prop:Wsmooth} For $s_1 \geq 0$, $s_2\geq -\frac12$, and  $a<\min(2s_1-s_2-\frac12,s_1-s_2)$, we have 
\[ \big\| \frac{\Delta}{\la\nabla\ra}( \overline{u_1}u_2) \big\|_{X_{W,\pm}^{ s_2+a, -b}} \lesssim   \|  \overline{u_1}u_2   \|_{X_{W,\pm}^{ s_2+1+a, -b}} \lesssim  \| u_1 \|_{X_S^{s_1, b}} \|u_2\|_{X^{s_1, b}_S}, \]
provided that $b<\frac12$ is sufficiently close to $\frac 12$. 
\end{prop}

 The following three propositions take care of the correction terms arising in Kato smoothing estimates for the Duhamel terms, see Proposition~\ref{prop:DKS} and Proposition~\ref{prop:DKW}. 
\begin{prop} \label{prop:Scsmooth} For $\frac12<s_1+a<\frac52$, and $0\leq a<\min(\frac12,\frac12+s_2,1+s_2-s_1)$, we have  
\[\|nu\|_{{X_S^{\frac12+,\frac{s_1+a}2-\frac34 }}} \lesssim \| u\|_{X^{s_1, b}_S} \| n \|_{X^{s_2, b}_{W,\pm}}, \]
provided that $b<\frac12$ is sufficiently close to $\frac 12$. 
\end{prop}

\begin{prop} \label{prop:Wcsmooth} For $\frac12 < s_2+a<\frac32$, and $0\leq  a< \min(2s_1-s_2-\frac12,s_1-s_2,  \tfrac{s_1+1}2-s_2)$, we have    
\[ \| \frac{\Delta}{\la\nabla\ra}(\overline{u_1}u_2) \|_{X_W^{ \frac12+, s_2+a-1}} \lesssim  \|   \overline{u_1}u_2   \|_{X_{W,\pm}^{   \frac32+, s_2+a-1}} \lesssim \| u_1 \|_{X_S^{s_1, b}} \|u_2\|_{X^{s_1, b}_S}, \] 
provided that $b<\frac12$ is sufficiently close to $\frac12$.
\end{prop} 
\begin{prop} \label{prop:Wcsmooth2} For $s_1>\frac18$, we have  
\[ \| \la \xi\ra^{\frac12+}\chi_{\lambda\ll\la \xi\ra}
  \widehat{\overline{u_1}u_2 }(\xi,\lambda)  \|_{L^2_{\xi,\lambda}} \lesssim \| u_1 \|_{X_S^{s_1, b}} \|u_2\|_{X^{s_1, b}_S}\]
provided that $b<\frac12$ is sufficiently close to $\frac12$.
\end{prop}
To prove Proposition~\ref{prop:Ssmooth}, by duality and the definition of $X^{s,b}$ spaces, it suffices to prove that
$$
\Big|\int_{\R\times \R^2} nu \bar v dt dx\Big|\les \| n \|_{X_{W,\pm}^{s_2, b}} \|u\|_{X_S^{s_1, b}}\|v\|_{X_S^{-s_1-a, b}}.
$$
By Plancherel's identity the left hand side is equal to
$$
\Big|\int_{\R^3\times \R^3} \widehat n(\zeta_1-\zeta_2) \widehat u(\zeta_2) \overline{\widehat v (\zeta_1)} d\zeta_2d\zeta_1  \Big|,
$$
where $\zeta_i=(\tau_i,\xi_i)\in\R\times \R^2$ are the corresponding Fourier variables. 

The claim follows from 
$$
\Big|\int_{\R^3\times \R^3} M(\zeta_1,\zeta_2) f(\zeta_1-\zeta_2) g_2(\zeta_2) g_1(\zeta_1)  d\zeta_2d\zeta_1  \Big|\les \|f\|_{L^2(\R^3)}\|g_1\|_{L^2(\R^3)}\|g_2\|_{L^2(\R^3)},
$$
where $f(\tau,\xi)=\widehat n(\tau,\xi)\la \xi\ra^{s_2}\la \tau\pm|\xi|\ra^b$, $g_1(\tau ,\xi )=\overline{\widehat {v}(\tau ,\xi )}  \la \xi \ra^{-s_1-a} \la \tau +|\xi |^2\ra^b$, $g_2(\tau ,\xi )= \widehat {u}(\tau ,\xi )  \la \xi \ra^{s_1} \la \tau +|\xi |^2\ra^b$, and 
$$
M(\zeta_1,\zeta_2)= \frac{\la\xi_1\ra^{s_1+a}}{\la  \xi_2\ra^{s_1} \la \xi_1-\xi_2\ra^{s_2} \la \tau_1-\tau_2\pm|\xi_1-\xi_2| \ra^b  \la \tau_1+|\xi_1|^2 \ra^b
\la \tau_2+|\xi_2|^2 \ra^b}.
$$

Following \cite{BHHT}, we define
$$
I(f,g_1,g_2):=\int_{\R^3\times \R^3} f(\zeta_1-\zeta_2)g_1(\zeta_1)g_2(\zeta_2)d\zeta_1d\zeta_2.
$$

 For dyadic values of $N$, $L$ greater than 1, let $\chi_W^{N,L}(\tau,\xi)$
 be the characteristic function of the set $$\mathfrak{W}^{N,L}_{\pm}=\{(\tau,\xi)\in\R \times \R^2:\la \xi\ra\sim N, \la \tau\pm |\xi|\ra\sim L\},$$ and 
 $\chi_S^{N,L}(\tau,\xi)$
be  the characteristic function of the set $$\mathfrak{S}^{N,L} = \{(\tau,\xi)\in \R \times \R^2:\la \xi\ra\sim N, \la \tau+ |\xi|^2\ra\sim L\}.$$
Note that when $f=f\chi_W^{N,L}$ and $g_i=g_i\chi_S^{N_i,L_i}$, we have $M\sim  \frac{N_1^{s_1+a}}{N_2^{s_1} N^{s_2} L^b L_1^b
L_2^b}$.  To prove Proposition~\ref{prop:Ssmooth}, it suffices to have the following estimate  for $I$ assuming that $f=f\chi_W^{N,L}$ and $g_i=g_i\chi_S^{N_i,L_i}$, and that they are $L^2$ normalized: 
\be\label{ssmooth}
|I(f,g_1,g_2)|\les \frac{N^{s_2}N_2^{s_1}}{N_1^{s_1+a}}(L_1L_2L)^b, 
\ee
for $s_1\geq 0 $, $s_2\geq -\frac12$, and $0\leq a<\min(\frac12,\frac12+s_2, 1-s_1+s_2)$, and 
provided that $b<\frac12$ is sufficiently close to $\frac 12$.  
 The claim of Proposition~\ref{prop:Ssmooth}  follows after summing in dyadic frequencies (using orthogonality and convolution structure for sum in $N$'s and since the range  of $b$ is open one can pull  an epsilon power  of $LL_1L_2$   in all estimates). 
 
In a similar fashion, Proposition~\ref{prop:Wsmooth} follows from the inequality  
\be\label{wsmooth}
|I(f,g_1,g_2)|\les \frac{N_1^{s_1}N_2^{s_1}}{N^{1+s_2+a}}(L_1L_2L)^b,
\ee
under the hypothesis of Proposition~\ref{prop:Wsmooth}. 

 Proposition~\ref{prop:Scsmooth} follows from  
\be\label{scsmooth}
|I(f,g_1,g_2)|\les \frac{N^{s_2}N_2^{s_1}}{N_1^{\frac12+}}(L_2L)^b L_1^{\frac34-\frac{s_1+a}{2}}.
\ee
And  Proposition~\ref{prop:Wcsmooth} follows from  
 \be\label{wcsmooth}
|I(f,g_1,g_2)|\les \frac{N_1^{s_1}N_2^{s_1}}{N^{\frac32+}}(L_1L_2)^b L^{1-s_2-a}.
 \ee
Finally, Proposition~\ref{prop:Wcsmooth2} follows from 
 \be\label{wcsmooth2}
|I(f,g_1,g_2)|\les \frac{N_1^{s_1}N_2^{s_1}}{N^{\frac12+}}(L_1L_2)^b 
 \ee
 under the condition that   $s_1>\frac18$ and $L\approx N$.
   
 To prove \eqref{ssmooth}, we need an angular Whitney decomposition as in \cite{BHHT}. Let, for fixed $A$ and $j$,  
 $$
 \mathfrak{Q}_j^A:=\Big\{\big(|\xi|\cos(\theta),|\xi|\sin(\theta)\big)\in\R^2: \theta  \text{ or } \theta+\pi \in \big [\frac\pi{A}(j-2),\frac\pi{A}(j+2)\big]  \Big\}.
 $$
 We also define 
 $$
 \mathfrak{S}^{N,L,A,j} := \{(\tau,\xi)\in \R \times \R^2:\la \xi\ra\sim N, \,\la \tau+ |\xi|^2\ra\sim L,  \,
 \xi\in \mathfrak{Q}_j^A\}.$$

 The following propositions are implicit in \cite{BHHT}, in the sense that, the claim is either identical to the claim of the corresponding proposition or can be obtained from the proof of it. 
 
 \begin{prop}\label{prop4.4} \cite[Proposition 4.4]{BHHT}  
 Assume that $1\ll N\les N_1\sim N_2, 64\leq A\ll N_1$, $16\leq |j_1-j_2|\leq 32$, and $\max(L_1,L_2,L)\les N_1^2$. Also assume that $f,g_1,g_2$ are $L^2$ normalized and are supported on the sets  $\mathfrak W^{N,L}$, $\mathfrak S^{N_k,L_k,A,j_k}, k=1,2,$ respectively, then
 $$|I|\les \frac{A^{\frac12}}{N_1} \sqrt{L_1L_2L}.$$
 \end{prop} 
  \begin{prop}\label{prop4.6} \cite[Proposition 4.6]{BHHT}  
 Assume that $1\ll N\les N_1\sim N_2, 64\leq A\ll N_1$, $16\leq |j_1-j_2|\leq 32$. Also assume that $f,g_1,g_2$ are $L^2$ normalized and are supported on the sets  $\mathfrak W^{N,L}$, $\mathfrak S^{N_k,L_k,A,j_k}, k=1,2,$ respectively, then
 $$|I|\les \frac{\sqrt{N_1} }{\sqrt{AN} } \frac{\sqrt{L_1L_2L}}{\sqrt{\max(L_1,L_2,L)}}.$$
 \end{prop} 
 In the propositions above, we have $I=0$ unless $A\gtrsim N_1/N$. 

   \begin{prop}\label{prop4.7} \cite[Proposition 4.7]{BHHT}    
 Assume that $1\ll N\les N_1\sim N_2,   A\sim N_1,  |j_1-j_2|\leq 16$. Also assume that $f,g_1,g_2$ are $L^2$ normalized and are supported on the sets  $\mathfrak W^{N,L}$, $\mathfrak S^{N_k,L_k,A,j_k}, k=1,2,$ respectively, then
 $$|I|\les \frac{1 }{\sqrt{N} } \frac{\sqrt{L_1L_2L}}{\sqrt{\max(L_1,L_2,L)}},$$
 and
 $$N_1\sim N  \text{ or } \max(L_1,L_2,L)\gtrsim NN_1.$$ 
 \end{prop} 
  \begin{cor}\label{cor:4446}
Under the hypothesis of Proposition~\ref{prop4.6} or Proposition~\ref{prop4.7}, we have
 $$|I|\les \min\Big( \frac{\sqrt{L_1L_2L}}{(N_1N\max(L_1,L_2,L))^{1/4}}  ,\frac{\sqrt{L_1L_2L}}{\sqrt{\max(L_1,L_2,L)}}\Big).$$
  \end{cor}
 The bounds above are immediate under the hypothesis of Proposition~\ref{prop4.7}. Also note  that the second bound above is immediate from Proposition~\ref{prop4.6} and the lower bound on $A$. The first bound is   the geometric mean of the bounds in  Proposition~\ref{prop4.4} and  Proposition~\ref{prop4.6} when $\max(L_1,L_2,L)\les N_1^2$, and it follows from Proposition~\ref{prop4.6} otherwise.
  \begin{prop}\label{prop4.8} \cite[Proposition 4.8]{BHHT} 
 Assume that $1\leq N_1\ll N_2$ or vice versa. Also assume that $f,g_1,g_2$ are $L^2$ normalized and are supported on the sets  $\mathfrak W^{N,L}$, $\mathfrak S^{N_k,L_k}, k=1,2,$ respectively, then
 $$N\sim \max(N_1,N_2) \text{ and } \max(L_1,L_2,L)\gtrsim \max(N_1^2,N_2^2), \text{ and } $$
 $$|I|\les  \frac{\sqrt{L_1L_2L}}{\sqrt{\max(L_1,L_2,L)}}.$$
 \end{prop} 
   \begin{prop}\label{prop4.9} \cite[Proposition 4.9]{BHHT} 
 Assume that $N\sim 1 $. Also assume that $f,g_1,g_2$ are $L^2$ normalized and are supported on the sets  $\mathfrak W^{N,L}$, $\mathfrak S^{N_k,L_k}, k=1,2,$ respectively, then
 $$|I|\les (L_1L_2L)^{\frac13}.$$
 \end{prop}  
 The  corollary below follows from Propositions~\ref{prop4.4}, \ref{prop4.6}, \ref{prop4.7},  and an angular Whitney decomposition as in \cite{BHHT}. We provide a proof since the bounds we need are not all stated in \cite{BHHT}.    
\begin{cor}\label{cor:smallN}
Assume that $1\ll N\les N_1\sim N_2  $. Also assume that $f,g_1,g_2$ are $L^2$ normalized and are supported on the sets  $\mathfrak W^{N,L}$, $\mathfrak S^{N_k,L_k}, k=1,2,$ respectively, then
 $$|I|\les (L_1L_2L)^bN_1^{-\frac12}N_1^{3-6b}, \,\,0<b\leq \frac12.$$
 Moreover, 
 we  have the bound 
 $$|I|\les N_1^{0+} \min\Big(  \frac{\sqrt{L_1L_2L}}{\sqrt{\max(L_1,L_2,L)}}, \frac{\sqrt{L_1L_2L}}{[N_1N\max(L_1,L_2,L)]^{\frac14} }\Big).$$ 
\end{cor}
\begin{proof}
Let $g_k^{A,j_k}(\tau,\xi)$ be $g_k(\tau,\xi)\chi_{\mathfrak{Q}_{j_k}^A}(\xi)$, $k=1,2$. 
 Fix a dyadic $M\geq 64$. We define an angular Whitney decomposition at height $M$ by 
 $$
 \R^2\times \R^2=\Big[\bigcup_{\stackrel{0\leq j_1,j_2\leq M-1}{  |j_1-j_2|\leq 16}}  \mathfrak{Q}_{j_1}^M\times  \mathfrak{Q}_{j_2}^M \Big]\bigcup\Big[\bigcup_{64\leq A\leq M, dyadic}\bigcup_{\stackrel{0\leq j_1,j_2\leq A-1}{16\leq |j_1-j_2|\leq 32}}  \mathfrak{Q}_{j_1}^A\times  \mathfrak{Q}_{j_2}^A\Big].
 $$
 By the Whitney decomposition above, at level $M=2^{-4} N_1$, we have 
$$
I(f,g_1,g_2)\les  \sum_{\stackrel{0\leq j_1,j_2\leq M-1}{  |j_1-j_2|\leq 16}}  I(f,g_1^{M,j_1},g_2^{M,j_2}) +\sum_{\stackrel{64\leq A\leq M}{dyadic}}\sum_{\stackrel{0\leq j_1,j_2\leq A-1}{16\leq |j_1-j_2|\leq 32}}  I(f,g_1^{A,j_1},g_2^{A,j_2}).
$$
We start with the first sum. By Proposition~\ref{prop4.7}, we have  
\be\label{tempcond}N_1\sim N  \text{ or } \max(L_1,L_2,L)\gtrsim NN_1, \ee and we can estimate the first sum above by
\begin{multline*} \les \frac{1 }{\sqrt{N} } \frac{\sqrt{L_1L_2L}}{\sqrt{\max(L_1,L_2,L)}} \sum_{\stackrel{0\leq j_1,j_2\leq M-1}{  |j_1-j_2|\leq 16}}  \|f\|_{L^2}\|g_1^{M,j_1}\|_{L^2}\|g_2^{M,j_2}\|_{L^2} \\ \les \frac{1 }{\sqrt{N} } \frac{\sqrt{L_1L_2L}}{\sqrt{\max(L_1,L_2,L)}} \|f\|_{L^2} \|g_1 \|_{L^2}\|g_2 \|_{L^2} 
\les (L_1L_2L)^bN_1^{-\frac12}N_1^{\frac32-3b}.
\end{multline*}
 In the second inequality, we used Cauchy-Schwarz inequality and almost orthogonality of $g_k^{M,j_k}$'s, and in the last inequality, we used \eqref{tempcond}. 

To estimate the second sum, we consider two cases i) $\max(L_1,L_2,L)\les N_1^2$ and ii) $\max(L_1,L_2,L)\gg N_1^2$.
In the first case, using Propositions \ref{prop4.4}, we estimate the second sum by
$$\|f\|_{L^2} \sqrt{L_1L_2L} \sum_{64\leq A\leq M}  \frac{A^{1/2}}{N_1} \sum_{\stackrel{0\leq j_1,j_2\leq A-1}{16\leq |j_1-j_2|\leq 32}}   \|g_1^{A,j_1}\|_{L^2}\|g_2^{A,j_2}\|_{L^2} 
$$
$$
\les  \sqrt{L_1L_2L} \sum_{64\leq A\leq M}  \frac{A^{1/2}}{N_1} 
 \les 
 \frac{\sqrt{L_1L_2L}}{N_1^{\frac12}  }\les (L_1L_2L)^bN_1^{-\frac12}N_1^{3-6b}.
 $$
 In the second case, using Proposition~\ref{prop4.6}, and the fact that $A\gtrsim N_1/N$, we similarly obtain
  $$
 |I|\les \frac{ \sqrt{L_1L_2L}}{ \sqrt{\max(L_1,L_2,L)}  }\les (L_1L_2L)^bN_1^{-1}N_1^{3-6b}.
 $$
 The second bound in the corollary follows similarly from Corollary~\ref{cor:4446}. The loss $N_1^{0+}$ is due to the summation over dyadic $A\les N_1$.  
\end{proof}

\begin{proof}[Proof of Proposition~\ref{prop:Ssmooth}]
As we discussed above, it suffices to prove \eqref{ssmooth},
$$
|I(f,g_1,g_2)|\les \frac{N^{s_2}N_2^{s_1}}{N_1^{s_1+a}}(L_1L_2L)^b, 
$$  
for $s_1 \geq 0$, $s_2\geq -\frac12$  and  $0\leq a<\min(\frac12,\frac12+s_2, 1-s_1+s_2)$, and 
provided that $b<\frac12$ is sufficiently close to $\frac 12$.  

In the case  $1\ll N\les N_1\sim N_2$, by Corollary~\ref{cor:smallN}, it suffices to check
$$
 (L_1L_2L)^bN_1^{-\frac12}N_1^{3-6b} \les     \frac{N^{s_2} }{N_1^{ a}}(L_1L_2L)^b.
$$
This holds for $a<\min(\frac12,\frac12+s_2)$ provided that $b$ is sufficiently close to $\frac12$.

In the case $1\leq N_2\ll N_1\sim N$, by Proposition~\ref{prop4.8}, we have 
$$
|I|\les \frac{\sqrt{L_1L_2L}}{\sqrt{\max(L_1,L_2,L)}} \les  \frac{(L_1L_2L)^b }{N_1^{6b-2}} \les  \frac{N^{s_2}N_2^{s_1} }{N_1^{s_1+ a}}(L_1L_2L)^b,$$
provided that $a<1+s_2-s_1$, $b$ is sufficiently close to $\frac12$, and $s_1\geq 0$.
In the second inequality we used $ \max(L_1,L_2,L)\gtrsim  N_1^2  $. The case $1\leq N_1\ll N_2\sim N$ is similar.

It remains to consider the case $N\sim 1$, which implies $N_1\sim N_2$.   Proposition~\ref{prop4.9} does not suffice to handle this case. 
We write (with $\zeta_i:=(\xi_i,\tau_i)\in\R^2\times \R$, $i=1,2$)
$$
\frac{|I|}{(L_1L_2L)^b} \les \int_{\R^3\times \R^3} \frac{\chi_{|\xi_1|\sim |\xi_2|\sim N_1} \chi_{|\xi_1-\xi_2|\les N} |f(\zeta_1-\zeta_2)|\ |g_1(\zeta_1)|\ |g_2(\zeta_2)|}{\la \tau_1+|\xi_1|^2\ra^b \la \tau_2+|\xi_2|^2\ra^b   \la \tau_1-\tau_2\pm|\xi_1-\xi_2|\ra^b}\ d\zeta_1d\zeta_2.
$$
It suffices to prove that the right hand side is $\les N_1^{-\frac12+}$. Letting $\zeta=\zeta_1-\zeta_2$ in the $\zeta_1$ integral, and by Cauchy-Schwarz inequality and the convolution structure, we bound the right hand side by the square root of 
\begin{multline*}
\sup_{|\xi_2|\sim N_1, \tau_2}\int_{\R\times \R^2 }\frac{  \chi_{|\xi |\les N}  }{\la \tau+\tau_2+|\xi+\xi_2|^2\ra^{2b} \la \tau_2+|\xi_2|^2\ra^{2b}   \la \tau \pm|\xi |\ra^{2b}}\ d\tau d\xi \\
\les \sup_{|\xi_2|\sim N_1, \tau_2}\int_{ |\xi|\les 1 }\frac{1  }{\la  \tau_2+|\xi+\xi_2|^2 \pm |\xi|\ra^{4b-1}  \la \tau_2+|\xi_2|^2\ra^{2b}  }\   d\xi \\
\les  \sup_{|\xi_2|\sim N_1}\int_{ |\xi|\les 1 }\frac{1  }{\la   |\xi+\xi_2|^2 -|\xi_2|^2  \ra^{4b-1}    }\   d\xi \les \int_{ |\xi|\les 1 }\frac{1  }{   |\xi|^{4b-1} N_1^{4b-1}    }\   d\xi \les N_1^{1-4b}.
\end{multline*}
\end{proof}

\begin{proof}[Proof of Proposition~\ref{prop:Wsmooth}] It suffices to prove  \eqref{wsmooth},
$$
|I(f,g_1,g_2)|\les \frac{N_1^{s_1}N_2^{s_1}}{N^{1+s_2+a}}(L_1L_2L)^b,
$$   for $a<\min(2s_1-s_2-\frac12,s_1-s_2)$    provided that $b<\frac12$ is sufficiently close to $\frac12$. 

In the case $1\ll N\les N_1\sim N_2$, this immediately follows from Corollary~\ref{cor:smallN}. In the cases $1\leq N_2\ll N_1\sim N$ and $1\leq N_1\ll N_2\sim N$, it follows from Proposition~\ref{prop4.8} as in the proof of Proposition~\ref{prop:Wsmooth}. Finally, the case $N\sim 1$ follows from Proposition~\ref{prop4.9}:
$$|I|\les (L_1L_2L)^{\frac13}\les (L_1L_2L)^b\les \frac{N_1^{s_1}N_2^{s_1}}{N^{1+s_2+a}}(L_1L_2L)^b$$
as $N\sim 1$ and $s_1\geq 0$.
\end{proof}
\begin{proof}[Proof of Proposition~\ref{prop:Scsmooth}] It suffices to prove   \eqref{scsmooth},
$$|I(f,g_1,g_2)|\les \frac{N^{s_2}N_2^{s_1}}{N_1^{\frac12+}}(L_2L)^b L_1^{\frac34-\frac{s_1+a}{2}}$$
 when $\frac12<s_1+a<\frac52$ and $0\leq a<\min(\frac12,\frac12+s_2,1+s_2-s_1)$ provided that $b<\frac12$ is sufficiently close to $\frac 12$.  Noting that the required bound follows from \eqref{ssmooth} when $L_1\les N_1^2$,  we can also assume that   $L_1\gg N_1^2$.

We first consider the range $\frac12<s_1+a<\frac32$. 
In the case $1\ll N\les N_1\sim N_2$, using the second claim of Corollary~\ref{cor:smallN} we have
$$
|I|\les  \frac{\sqrt{ L_1L_2L}}{\sqrt{\max(L_1,L_2,L)}}.
$$
It suffices to note that
$$
 \frac{\sqrt{ L_1L_2L}}{\sqrt{\max(L_1,L_2,L)}}\les  \frac{ (L_2L)^b  \sqrt{L_1}}{\max(L_1,L_2,L)^{2b-\frac12} }\les   (L_2L)^b   L_1^{1-2b}   \les  \frac{N^{s_2}N_2^{s_1}}{N_1^{\frac12+}}(L_2L)^b L_1^{\frac34-\frac{s_1+a}{2}}
$$
  provided that $a<\frac12$ and $b$ is sufficiently close to $\frac12$.

 In the cases $1\leq N_2\ll N_1\sim N$ and $1\leq N_1\ll N_2\sim N$, by Proposition~\ref{prop4.8} we have
$$
|I|\les  \frac{\sqrt{ L_1L_2L}}{\sqrt{\max(L_1,L_2,L)}}\les  (L_2L)^b   L_1^{1-2b}   \les  \frac{N^{s_2}N_2^{s_1}}{N_1^{\frac12+}}(L_2L)^b L_1^{\frac34-\frac{s_1+a}{2}},
$$ 
 for $a<s_2-s_1+1$ provided that $b$ is sufficiently close to $\frac12$.
 
 In the case $N\sim1$, as in the proof of Proposition~\ref{prop:Ssmooth} above, we have
 $$
 \frac{|I|}{(L_2L)^b L_1^{\frac34-\frac{s_1+a}2}}\les \sup_{|\xi_2|\sim N_1, \tau_2} \sqrt{ \int_{\R\times \R^2 }\frac{  \chi_{|\xi |\les N}  }{\la \tau+\tau_2+|\xi+\xi_2|^2\ra^{\frac32- s_1-a  } \la \tau_2+|\xi_2|^2\ra^{2b}   \la \tau \pm|\xi |\ra^{2b}}\ d\tau d\xi}
 $$ 
$$\les \sup_{|\xi_2|\sim N_1 } \sqrt{ \int_{ |\xi |\les 1 }\frac{ 1  }{\la  |\xi+\xi_2|^2 -|\xi_2|^2\ra^{2b+\frac12- s_1-a  } }\   d\xi} 
\les \frac1{N_1^{b+\frac14-\frac{s_1+a}2}} \les N_1^{s_1-\frac12-},
$$
which holds for $a<\frac12$.

We now consider the range $\frac32\leq s_1+a<\frac52$. We have   two cases: $N_1\sim\max(N,N_2)$ and $N_1\ll N_2\sim N$.

In the former case, since $L_1\gg N_1^2 \gtrsim  |\tau_1+|\xi_1|^2 - \tau_2-|\xi_2|^2   - ( \tau_1-\tau_2\pm|\xi_1-\xi_2|)|$,
we have $L_1\les \max(L,L_2)$. Without loss of generality assume $L_1\les L$. We have 
$$
\frac{|I|}{(L_2L)^b L_1^{\frac34-\frac{s_1+a}2} }  \les \frac{|I|}{L_2^bL^{0+} L_1^{b+\frac34-\frac{s_1+a}2-}  } \les \frac{|I|}{L_2^bL^{0+} N_1^{2b+\frac32- s_1-a - }}$$
$$ \les N_1^{-2b-\frac32+ s_1+a + }\sup_{|\xi_1|\sim N_1, \tau_1+|\xi_1|^2\sim L_1} \sqrt{ \int_{\R\times \R^2 }\frac{  \chi_{|\xi_2|\sim N_2} \chi_{|\xi_1-\xi_2|\sim N}   }{ \la \tau_2+|\xi_2|^2\ra^{2b}   \la \tau_1-\tau_2 \pm|\xi_1-\xi_2 |\ra^{0+}}\ d\tau_2 d\xi_2}
$$ 
$$\les N_1^{-2b-\frac32+ s_1+a + } \sup_{|\xi_1|\sim N_1,\tau_1+|\xi_1|^2\sim L_1} \sqrt{ \int_{  \R^2 }  \chi_{|\xi_2|\sim N_2} \chi_{|\xi_1-\xi_2|\sim N}     \ d\xi_2}
$$ 
$$
\les  N_1^{  s_1+a -\frac52+ }    \min(N,N_2)\les \frac{N^{s_2}N_2^{s_1}}{N_1^{\frac12+}},
$$
which holds under the conditions of the proposition.

In the latter case, if $L_1\les \max(L,L_2)$ the argument above remains valid. In the case $L_1\gg \max (L,L_2)$, we have $L_1\sim N_2^2$. Therefore, we have 
$$
\frac{|I|}{(L_2L)^b L_1^{\frac34-\frac{s_1+a}2} } \les N_2^{s_1+a-\frac32  } \sup_{|\xi_1|\sim N_1, \tau_1+|\xi_1|^2\sim L_1} \sqrt{ \int_{\R\times \R^2 }\frac{  \chi_{|\xi_2|\sim N_2} \chi_{|\xi_1-\xi_2|\sim N}   }{ \la \tau_2+|\xi_2|^2\ra^{2b}   \la \tau_1-\tau_2 \pm|\xi_1-\xi_2 |\ra^{2b}}\ d\tau_2 d\xi_2}
$$
$$\les  N_2^{s_1+a-\frac32  }\sup_{|\xi_1|\sim N_1,\tau_1+|\xi_1|^2\sim L_1} \sqrt{ \int_{  \R^2 }\frac{  \chi_{|\xi_2|\sim N_2} \chi_{|\xi_1-\xi_2|\sim N}   }{   \la \tau_1+|\xi_2|^2\pm|\xi_1-\xi_2 |\ra^{4b-1}}\ d\xi_2}\les N_2^{s_1+a-\frac32+2-4b}
$$
which suffices. The last inequality follows from Lemma~\ref{lem:cv}.
\end{proof}

\begin{proof}[Proof of Proposition~\ref{prop:Wcsmooth}] 
Once again, it suffices to prove \eqref{wcsmooth}:  
 $$
|I(f,g_1,g_2)|\les \frac{N_1^{s_1}N_2^{s_1}}{N^{\frac32+}}(L_1L_2)^b L^{1-s_2-a} 
 $$
when $\frac12 < s_2+a<\frac32$,  $0\leq  a< \min(2s_1-s_2-\frac12,s_1-s_2,  \tfrac{s_1+1}2-s_2)$,  
provided that $b<\frac12$ is sufficiently close to $\frac12$. Note that this follows from \eqref{wsmooth} if  $L\les N$, and hence we can assume $L\gg N$.
 
 First consider the case $\frac12<s_2+a<1$. 
 In the case $1\ll N\les N_1\sim N_2$, using the second claim of Corollary~\ref{cor:smallN}, we have
 $$|I(f,g_1,g_2)|\les  \frac{\sqrt{L_1L_2L}}{\sqrt{\max(L_1,L_2,L)}}\les  \frac{ (L_1L_2)^b L^{1-s_2-a}}{\max(L_1,L_2,L)^{2b-s_2-a}}
  $$
   $$
\les \frac{   (L_1L_2)^b L^{1-s_2-a}}{N^{ 2b-s_2-a} }\les \frac{N_1^{2s_1} }{N^{\frac32+}}(L_1L_2)^b L^{1-s_2-a},
 $$
 provided that    $a<2s_1-s_2-\frac12$ and that $b<\frac12$ is sufficiently close to $\frac12$. 
 
 In the case $N\sim 1$, we have
 $$
 \frac{|I|}{(L_1L_2)^b L^{1- s_2-a}  }\les \sup_{|\xi_2|\sim N_1, \tau_2} \sqrt{ \int_{\R\times \R^2 }\frac{  \chi_{|\xi |\les N}  }{\la \tau+\tau_2+|\xi+\xi_2|^2\ra^{ 2b  } \la \tau_2+|\xi_2|^2\ra^{2b}   \la \tau \pm|\xi |\ra^{2-2 s_2-2a}}\ d\tau d\xi}.
 $$ 
Note that the $\tau$ integral is $\les 1$ provided that  $2-2s_2-2a+2b>1$, and since $N\sim 1$, the $\xi$ integral is bounded as well. Therefore, the integral above is $\les 1\les  \frac{N_1^{s_1}N_2^{s_1}}{N^{\frac32+}}$. 

In the case $1\leq N_2\ll N_1\sim N$ (the case $1\leq N_1\ll N_2\sim N$ follows by symmetry), we have 
 $$
 \frac{|I|}{(L_1L_2)^b L^{1- s_2-a}  }\les \frac1{N^{1-s_2-a}} \sup_{|\xi|\sim N, \tau} \sqrt{ \int_{\R\times \R^2 }\frac{  \chi_{|\xi_2 |\les N_2}  }{\la \tau+\tau_2+|\xi+\xi_2|^2\ra^{2b } \la \tau_2+|\xi_2|^2\ra^{2b}   }\ d\tau_2 d\xi_2} 
 $$ 
 $$
 \les \frac1{N^{1-s_2-a}}  \sup_{|\xi|\sim N, \tau} \sqrt{ \int_{ \R^2 }\frac{  \chi_{|\xi_2 |\les N_2}  }{\la \tau-|\xi_2|^2+|\xi+\xi_2|^2\ra^{ 4b-1  }   }\   d\xi_2} 
 $$
 $$
 =\frac1{N^{1-s_2-a}}  \sup_{|\xi|\sim N, \tau} \sqrt{ \int_{ \R^2 }\frac{  \chi_{|\xi_2 |\les N_2}  }{\la \tau+|\xi|^2+2 \xi\cdot \xi_2 \ra^{ 4b-1  }   }\   d\xi_2} \les  \frac1{N^{1-s_2-a}} \frac{N_2^{\frac32-2b }}{N^{2b-\frac12 }}\les N^{s_1-\frac32-}N_2^{s_1},
 $$
 provided that $a<s_1-s_2$ and $b$ is sufficiently close to $\frac12$. In the second ineqality we used $L\gtrsim N$, and in the second to last inequality we used Lemma~\ref{lem:cv}.
 
 We now consider the case $1\leq s_2+a<\frac32$. By symmetry, it suffices to consider the case  $N_2, N\les N_1 $.
 We have
  $$
 \frac{|I|}{(L_1L_2)^b L^{1- s_2-a}  }\les \sup_{|\xi|\sim N, \tau} \sqrt{ \int_{\R\times \R^2 }\frac{  \chi_{|\xi_2 |\les N_2}    \la \tau \pm|\xi |\ra^{  2 s_2+2a -2}}{\la \tau+\tau_2+|\xi+\xi_2|^2\ra^{2b } \la \tau_2+|\xi_2|^2\ra^{2b} }\ d\tau_2 d\xi_2} 
 $$ 
  $$
 \les   \sup_{|\xi|\sim N, \tau} \sqrt{ \int_{ \R^2 }\frac{  \chi_{|\xi_2 |\les N_2}   \la \tau \pm|\xi |\ra^{  2 s_2+2a -2} }{\la \tau-|\xi_2|^2+|\xi+\xi_2|^2\ra^{ 4b-1  }   }\   d\xi_2} 
 $$
   $$
 \les   \sup_{|\xi|\sim N, \tau} \sqrt{ \int_{ \R^2 }\frac{  \chi_{|\xi_2 |\les N_2}   N_1^{  4 s_2+4a -4} }{\la \tau-|\xi_2|^2+|\xi+\xi_2|^2\ra^{ 4b-1  }   }\   d\xi_2 +\int_{ \R^2 }\frac{  \chi_{|\xi_2 |\les N_2}     }{\la \tau-|\xi_2|^2+|\xi+\xi_2|^2\ra^{ 4b+1-2s_2-2a  }   }\   d\xi_2 } 
 $$
    $$
 \les  N_2^{\frac32-2b}N^{-2b+\frac12}N_1^{2s_2+2a-2} + N_2^{s_2+a-2b+\frac12}  N^{s_2+a-2b-\frac12}\les  \frac{N_1^{s_1}N_2^{s_1}}{N^{\frac32+}}
 $$
 provided that $s_2+a<\min(s_1,\tfrac{s_1+1}2)$  and $b$ is sufficiently close to $\tfrac12$. We used Lemma~\ref{lem:cv} to estimate  the integrals in $\xi_2$. 
\end{proof}
\begin{proof}[Proof of Proposition~\ref{prop:Wcsmooth2}]
In the case $1\ll N\les N_1\approx  N_2$, by the last bound in Corollary~\ref{cor:smallN}, we have (using $L\approx N$)
$$
|I|\les   \frac{\sqrt{L_1L_2L}}{[N_1N\max(L_1,L_2,L)]^{\frac14} }\les (L_1L_2)^bN^{1-2b}N_1^{-\frac14}\les (L_1L_2)^b\frac{N_1^{2s_1}}{N^{\frac12+}}
$$
provided that $s_1>\frac18$ and $b$ is sufficiently close to $\frac12$. 

In the case $1\leq N_2\ll N_1\approx N$, by Proposition~\ref{prop4.8}, we have 
$$|I|\les   \frac{\sqrt{L_1L_2L}}{ \max(L_1,L_2,L)^{\frac12} }\les (L_1L_2)^b N^{-\frac12+}\les (L_1L_2)^b\frac{N_1^{ s_1} N_2^{s_1} }{N^{\frac12+}}
$$
provided that $s_1>0$ and $b$ is sufficiently close to $\frac12$. In the second inequality we used the facts that $L\approx N$ and $\max(L_1,L_2,L)\gtrsim N^2$. 

In the case $L\approx N\sim 1$, we have the bound $(L_1L_2L)^{\frac13}\les L_1^bL_2^b$, which suffices for $s_1\geq 0$.
\end{proof}

 \section{Local well--posedness and smoothing} \label{sec:lwp}
The assertion of Theorem~\ref{thm:local} follows from the a priori linear estimates established in Section~\ref{sec:lin} and the nonlinear estimates in Section~\ref{sec:nonlin}. The proof follows along the lines of the proof of Theorem 1.3 in Section 5 of \cite{etza} (also see the proof of Theorem 2.4 in Section 4 of \cite{etnls}). In particular, the fix point argument for equation \eqref{eq:duhamel} in the $X^{s,b}$ space
for the linear solution follows from   Propositions~\ref{prop:SXsb} and ~\ref{prop:KGXsb}, and the Kato smoothing bounds in Propositions~\ref{prop:KSS} and ~\ref{prop:KSKG}.  For the nonlinear terms we use Propositions~\ref{prop:Ssmooth}--\ref{prop:Wcsmooth2} and the properties \eqref{eq:xs1}, \eqref{eq:xs2}, \eqref{eq:xs3} of $X^{s,b}$ spaces. 
Similarly, the   Schr\"odinger and Klein-Gordon parts of the solution   belong to 
$  C^0_tH^{s_1}_{x,y}([0,T]\times U) \cap C^0_y\cH^{s_1}_S(\R^+ \times \R\times [0,T]) $  and $  C^0_tH^{s_2}_{x,y}([0,T]\times U) \cap C^0_y\cH^{s_2}_W(\R^+ \times \R\times [0,T]) $, respectively, using  Propositions~\ref{prop:KSS}, \ref{prop:KSKG}, \ref{prop:tracelem}, \ref{prop:DKS}, and \ref{prop:DKW}. These a priori estimates also imply continuous dependence  on initial data, see  Section 5 of \cite{etza}. 
Note that the solution is unique once we fix an extension of the initial data, however it is not clear whether the restriction of the solution to the half plane is independent of the extension.

Finally, the smoothing bound in Theorem~\ref{thm:smooth} follows from the same estimates as in the proof of Theorem 1.1 in  Section 5 of \cite{etnls}. 
 
\section{Appendix}

Below, we list four lemmas that are used repeatedly throughout the paper.   For the proof of the first lemma see  the Appendix of \cite{etza}. The second lemma is the weighted Schur's test,  and the last two are  elementary calculus lemmas.  
\begin{lemma}\label{lem:sums}   If  $\beta\geq \gamma\geq 0$ and $\beta+\gamma>1$, then
\be\nn
\int_\R \frac{1}{\la \tau-k_1\ra^\beta \la \tau-k_2\ra^\gamma} d\, \tau \lesssim \la k_1-k_2\ra^{-\gamma} \phi_\beta(k_1-k_2),
\ee
where
 \be\nn
\phi_\beta(k):=\sum_{|n|\leq |k|}\frac1{\la n\ra^\beta}\sim \left\{\begin{array}{ll}
1, & \beta>1,\\
\log(1+\la k\ra), &\beta=1,\\
\la k \ra^{1-\beta}, & \beta<1.
 \end{array}\right.
\ee
The statement remains valid when $\la \tau-k_2\ra$ is replaced with $| \tau-k_2|$ provided that $\gamma<1$.
\end{lemma} 

\begin{lemma}\label{lem:schur} Let $T$ be an integral operator  with kernel $K(\theta,\eta)$, $\theta\in\R^m$, $\eta\in\R^n$. Assume that for some positive functions $p(\theta)$, $q(\eta)$, and some constants $A, B$ we have 
$$
\int |K(\theta,\eta)| p(\theta) d\theta \leq A q(\eta), \text{ for a.e. } \eta,
$$ 
$$
\int |K(\theta,\eta)| q(\eta) d\eta \leq B p(\theta), \text{ for a.e. } \theta,
$$ 
then $\|T\|_{L^2\to L^2}\leq \sqrt {AB}$.
\end{lemma}

\begin{lemma}\label{lem:chvar} For any $f \in L^1$, $g \in C^1$, any measurable $h$, and any $K \in L^2$, we have
$$\Big\|\int f(\tau-g(\eta))  h(\eta)  K(\eta)d\eta \Big\|_{L^2_\tau}\les \|f\|_{L^1} \|K\|_{L^2} \sup_{\eta}\frac{|h(\eta)|}{\sqrt{|g^\prime(\eta)|}}.$$
\end{lemma}
\begin{proof}
Change variables in the $\eta$ integral by setting $\rho = g(\eta)$. Then apply Young's inequality and undo the change of variable.  
\end{proof}
\begin{lemma}\label{lem:cv}
Fix  $0\leq \alpha<1$  and $M,N\gtrsim 1$. Assume that $f:\R^2\to \R$ is $C^1$ and $|f_x|\gtrsim N$. We have
$$
  \int_{\R^2}\frac{\chi_{|x|,|y|<M}}{\la f(x,y)\ra^\alpha}dx dy\les M^{2-\alpha  }N^{-\alpha }.
$$
\end{lemma} 
\begin{proof}
It suffices to prove that 
$$
\sup_y  \int_{\R }\frac{\chi_{|x| <M}}{\la f(x,y)\ra^\alpha}dx \les M^{1-\alpha  }N^{-\alpha }.
$$
The contribution of the set $\{x: \la f(x,y)\ra> MN\}$ is $\leq 2 M(MN)^{-\alpha}$, which is the needed bound. 
Also note that the measure of the set $\{x:  \la f(x,y)\ra\sim \frac{MN}{2^k}\}$ is $\les  M2^{-k}$, for $k=0,..,\log(MN)$. The claim follows by summing the contributions over $k$.   
\end{proof}

\end{document}